\documentclass[12pt]{article}

\usepackage{mathtools}
\usepackage{mathrsfs}
\usepackage{bbm}
\usepackage{graphicx}
\usepackage{amsmath}
\usepackage{amssymb}
\usepackage{amsfonts}
\usepackage{epsfig}
\usepackage{graphicx,float}
\usepackage{color}
\usepackage{bbm}
\usepackage[table]{xcolor}
\usepackage[margin=1in]{geometry}
\usepackage{enumitem}
\usepackage{setspace}
\usepackage{algorithm}
\usepackage{algorithmic}
\usepackage{enumerate}
\usepackage{hyperref}
\newcommand{\footremember}[2]{%
    \footnote{#2}
    \newcounter{#1}
    \setcounter{#1}{\value{footnote}}%
}
\newcommand{\footrecall}[1]{%
    \footnotemark[\value{#1}]%
}
\usepackage{endnotes}
\let\footnote=\endnote
\let\enotesize=\normalsize
\def\notesname{Endnotes}%
\def\makeenmark{$^{\theenmark}$}
\def\enoteformat{\rightskip0pt\leftskip0pt\parindent=1.75em
  \leavevmode\llap{\theenmark.\enskip}}

\usepackage{natbib}
 \bibpunct[, ]{(}{)}{,}{a}{}{,}%
 \def\bibfont{\small}%
 \def\bibsep{\smallskipamount}%
 \def\bibhang{24pt}%
\newtheorem{lemma}{Lemma}[section]
\newtheorem{corollary}{Corollary}[section]
\newtheorem{theorem}{Theorem}[section]
\newtheorem{proposition}{Proposition}[section]

\newcommand{\beq}{\begin{equation}}
\newcommand{\eeq}{\end{equation}}
\newcommand{\beqa}{\begin{eqnarray}}
\newcommand{\eeqa}{\end{eqnarray}}
\newcommand{\beqas}{\begin{eqnarray*}}
\newcommand{\eeqas}{\end{eqnarray*}}
\newcommand{\ei}{\end{itemize}}
\newcommand{\ba}{\begin{array}}
\newcommand{\ea}{\end{array}}
\newcommand{\bn}{\begin{eqnarray}}
\newcommand{\en}{\end{eqnarray}}
\newcommand{\bns}{\begin{eqnarray*}}
\newcommand{\ens}{\end{eqnarray*}}
\newcommand{\E}{{\mathbb{E}}}
\newcommand{\bbr}{\Bbb{R}}
\newcommand{\bbe}{\Bbb{E}}

\newcommand{\nn}{\nonumber}
\def\eqnok#1{(\ref{#1})}
\def\argmin{{\rm argmin}}

\def\vgap{\vspace*{.1in}}

\begin{document}


%\RUNAUTHOR{Ghadimi, Perkin, and Powell}

%\RUNTITLE{Parametric Simulation Optimization for Multistage Stochastic Programming}

\title{Reinforcement Learning via Parametric Cost Function Approximation for Multistage Stochastic Programming}

%\ARTICLEAUTHORS{%
\author{Saeed Ghadimi\footremember{uni}{Department of Operations Research and Financial Engineering, Princeton University}\thanks{email: sghadimi@princeton.edu.}
\and
Raymond T. Perkins \footrecall{uni}\thanks{email: raymondp@princeton.edu.}
\and Warren B. Powell \footrecall{uni}\thanks{email: powell@princeton.edu}\\
\date{August 2019}
%\AUTHOR{Marg Arinella}
%\AFF{Institute for Food Adulteration, University of Food Plains, Food Plains, MN 55599, \EMAIL{m.arinella@adult.ufp.edu}}
} % end of the block

%\title{Parametric Cost Function Approximation for Multistage Stochastic Optimization with an Energy Storage Application}
%%%%%%%%%%%%%%%%%%%%%%%%%%%%
%%%%%%%%%%%%%%%%%%%%%%%%%%%%
%\author[SG]{Saeed Ghadimi}
%\author[RP3]{Raymond T. Perkins III}
%\author[WBP]{Warren B. Powell}
%
%
%\address[SG]{sghadimi@princeton.edu}
%\address[RP3]{raymondp@princeton.edu}
%\address[WBP]{powell@princeton.edu}
%%%%%%%%%%%%%%%%%%%%%%%%%%%
%%%%%%%%%%%%%%%%%%%%%%%%%%%

\maketitle

{\bf{Abstract:}} The most common approaches for solving stochastic resource allocation problems in the research literature is to either use value functions (``dynamic programming") or scenario trees (``stochastic programming") to approximate the impact of a decision now on the future.  By contrast, common industry practice is to use a deterministic approximation of the future which is easier to understand and solve, but which is criticized for ignoring uncertainty.  We show that a parameterized version of a deterministic lookahead can be an effective way of handling uncertainty, while enjoying the computational simplicity of a deterministic lookahead.  We present the parameterized lookahead model as a form of policy for solving a stochastic base model, which is used as the basis for optimizing the parameterized policy.  This approach can handle complex, high-dimensional state variables, and avoids the usual approximations associated with scenario trees.  We formalize this approach and demonstrate its use in the context of a complex, nonstationary energy storage problem.


\section{Introduction}
There has been a long history in industry of using deterministic optimization models to make decisions that are then implemented in a stochastic setting.  Energy companies use deterministic forecasts of wind, solar and loads to plan energy generation (\cite{wallace2003stochastic}); airlines use deterministic estimates of flight times to schedule aircraft and crews (\cite{lan2006planning}); and retailers use deterministic estimates of demands and travel times to plan inventories (\cite{harrison1999multi}).  These models have been widely criticized in the research community for not accounting for uncertainty, which often motivates the use of large-scale stochastic programming models which explicitly model uncertainty in future outcomes (\cite{mulvey1995robust} and \cite{birge2011introduction}). These large-scale stochastic programs, which use scenario trees to approximate future events, have been applied to unit commitment (\cite{jin2011modeling}), hydroelectric planning (\cite{carpentier2015managing}), and transportation (\cite{lium2009study}). These models are computationally very demanding, yet still require a number of approximations and, as a result do not produce optimal policies.

We make the case that these previous approaches ignore the true problem that is being solved, which is always stochastic.  The so-called ``deterministic models'' used in industry are almost always parametrically modified deterministic approximations, where the modifications are designed to handle uncertainty.  Both the ``deterministic models'' and the ``stochastic models'' (formulated using the framework of stochastic programming) are examples of lookahead policies to solve a stochastic optimization problem with the goal of finding the best policy which is typically tested using a simulator, but which may be field tested in an online environment (the real world).

In this paper, we characterize these modified deterministic models as {\it parametric cost function approximations} (CFAs) which puts them into the same category as other parameterized policies that are well known in the research community working on policy search (\cite{ng2000pegasus}, \cite{peshkin2000learning}, \cite{hu2007evolutionary}, \cite{Robots}, and \cite{mannor2003cross}).  A parallel community has evolved under the name simulation-optimization (see the recent edited volume \cite{fu2015handbook}), where powerful tools have been developed based on the idea of taking derivatives of simulations (see the extensive literature on derivatives of simulations covered in \cite{glasserman1991gradient}, \cite{ho1992discrete}, \cite{kushner2003stochastic}, \cite{cao2008stochastic});  a nice tutorial is given in \cite{chau2014simulation}. Much of this literature focuses on derivatives of discrete event simulations with respect to static parameters such as a buffer stock.  Our problem also exhibits static parameters, but in the form of parameterized modifications of constraints in a policy that involves solving a linear program.

Our use of modified linear programs is new to the policy search literature, where ``policies'' are typically parametric models such as linear models (``affine policies''), structured nonlinear models (such as (s,S) policies for inventories) or neural networks~(\cite{HanE16}). There are two dimensions to our approach: the design of the parameterized lookahead model, and the optimization of the parameters so that a particular parameterization performs as well as it can. The process of designing the modifications (in this paper, these modifications always appear in the constraints) requires the same art as the design of any statistical model or parametric policy, a process that requires exploiting the structure of the problem. This paper addresses the second dimension, which is the design of gradient-based search algorithms which are nontrivial in this setting.

This paper formalizes the idea, used for years in industry, that an effective way to solve complex stochastic optimization problems is to shift the modeling of uncertainty from a lookahead approximation, where even deterministic lookahead models can be hard to solve, to the stochastic base model, typically implemented as a simulator but which might also be the real world. Tuning a model in a stochastic simulator makes it possible to handle arbitrarily complex dynamics, avoiding the many approximations (such as two-stage models, scenario trees, exogenous information that is independent of decisions) that are standard in stochastic programming. In the process, we are expanding the range of problems considered by the simulation-optimization community to the entire class of vector-valued multistage stochastic optimization problems considered in the stochastic programming literature.

The CFAs make it possible to exploit structural properties.  For example, it may be obvious that the way to handle uncertainty when planning energy generators in a unit commitment problem is to require extra reserves at all times of the day.  A stochastic programming model encourages this behavior, but the requirement for a manageable number of scenarios will produce the required reserve only when one of the scenarios requires it. In addition, the almost universal use of two-stage approximations (where the future is modeled as a single stage, which means that decisions in the future are allowed to see into the future) underestimates the effect of uncertainty in the future. Imposing a reserve constraint (which is a kind of cost function approximation) allows us to impose the aforementioned requirement at all times of the day, and to tune it under very realistic conditions without any of the approximations required by stochastic programs.

Designing a parametric CFA closely parallels the design of any parametric statistical model, which is part art (creating the model) and part science (fitting the model). In fact, the design of the parameterization and the tuning of the parameters each represent important areas of research.
%To illustrate the process of designing a parametric cost function approximation, we use the setting of a time-dependent stochastic inventory planning problem that arises in the context of energy storage, but this could arise in any inventory planning setting.  We assume we have access to rolling forecasts where forecast errors are based on careful modeling of actual and predicted values for energy loads, generation from renewable sources, and prices. The combination of the time-dependent nature and the availability of rolling forecasts which are updated each time period make this problem a natural setting for lookahead models, where the challenge is how to handle uncertainty.  We have selected this problem since it is relatively small, simplifying the extensive computational work.  However, our methodology is scalable to any problem setting which is currently being solved using a deterministic model.
Most important, the parametric CFA opens up a fundamentally new approach for providing practical tools for solving high-dimensional, stochastic programming problems.  It provides an alternative to classical stochastic programming with its focus on optimizing a stochastic lookahead model which requires a variety of approximations to make it computationally tractable (and even then, it is typically computationally very demanding).
%The parametric CFA makes it possible to incorporate problem structure, such as the recognition that robust solutions can be achieved using standard methods such as schedule slack and/or buffer stocks.

The parametric CFA makes it possible to incorporate problem structure for handling uncertainty. Some examples include: supply chains handle uncertainty by introducing buffer stocks; hospitals can handle uncertainty in blood donations and the demand for blood by maintaining supplies of O-minus blood, which can be used by anyone; and grid operators handle uncertainty in generator failures, as well as uncertainty in energy from wind and solar, by requiring generating reserves. Central to our approach is the ability to manage uncertainty by recognizing effective strategies for responding to unexpected events.  We would argue that this structure is apparent in many settings, especially in complex resource allocation problems. At a minimum, we offer that our approach represents an interesting, and very practical, alternative to stochastic programming.

This paper makes the following contributions. First, we introduce and develop the idea of parameterized cost function approximation as a tool for solving important classes of multistage stochastic programming problems, shifting the focus from solving complex, stochastic lookahead models to optimizing a stochastic base model. This approach is computationally comparable to solving deterministic approximations, with the exception that the parametric modifications have to be optimized, typically in a simulator that avoids the many approximations made in stochastic lookahead models. To the best of our knowledge, this is the first fundamentally new approach for solving the aforementioned class of stochastic optimization problems which is as easy to compute  as a deterministic lookahead.

%Secondly, we present simulation-based approximation algorithms to optimize the introduced parameters. Thirdly, we illustrate different styles of parametric approximations using the context of a nonstationary energy inventory problem, and quantify the benefits over a basic deterministic lookahead without adjustments.
Second, we provide some theoretical results about the structure of the objective function and the optimal policy of our optimization problem in the CFA approach. In particular, we show that while the objective function is generally nonconvex, it admits special structure under different types of parameterization policies.

Third, we propose a simulation-based optimization algorithm and establish its finite-time rate of covergence for performing the policy search in our nonconvex stochastic optimization problem. Since this algorithm only uses noisy objective values, it gives us more flexibility in choosing the parametric model for the CFA approach.

Finally, we specialize our CFA approach for a complex, nonstationary energy storage problem in the presence of rolling forecasts. Our numerical experiments demonstrate that while our optimization problem is nonconvex, our proposed algorithm, regardless of the quality of the starting point, can find a parameterization policy whose performance is significantly better than a base policy using unmodified rolling forecasts.
%We also show that starting from the optimal parameter values if forecasts were perfect represents a robust starting point for problems with uncertain forecasts. It should be pointed out that the appropriate form of the parameterization in our approach mainly depends on the application which would be supplied by domain experts. For example, in this paper we parameterize forecasts by multiplying them times a tunable parameter that can be optimized to handle uncertainty in the forecasts. Simulation-based optimization algorithms, including the ones proposed in this paper, can then be used to optimize the values of parameters.

Our presentation is organized as follows. The modeling framework and an overview of the different classes of policies are given in Section 2. We then formally introduce the parametric CFA approach in Section 3. Algorithms for optimizing policy parameters together with their convergence properties and some theoretical results about structure of the optimization problem in the CFA approach are presented in Section 4. We then specialize the parametric CFA approach for an energy storage application and present some numerical experiments for solving this problem in Section 5. Finally, we conclude the paper in Section 6.
%%%%%%%%%%%%%%%%%%%%%%%%%%%
%%%%%%%%%%%%%%%%%%%%%%%%%%%
\section{Canonical Model and Solution Strategies}
Our main problem of interest in this paper is to find a policy $\pi$ that solves
\begin{equation} \label{objective1}
	\min_{\pi \in \Pi} \: \: \mathbb{E}^{\pi} \bigg[ \sum^T_{t=0} C_t(S_t, X^{\pi}_t(S_t)) \: \bigg| \: S_0  \bigg],
\end{equation}
where $S_{t+1} = S^M(S_t,X^\pi_t(S_t),W_{t+1})$. Here, $\Pi$ denotes the set of all possible policies, $S_t$ represents the sate variable, and $X^\pi_t(S_t)$ denotes the decision function (policy) which determines decision variable $x_t\in \mathcal{X}_t \subseteq \bbr^n$. Furthermore, $W_{t+1}$ denotes exogenous information which describes the information that first becomes known at time $t+1$ which may depend on the state $S_t$ and/or the action $x_t$, and $S^M(\cdot)$ denotes the transition function which explicitly describes the relationship between the state of the model at time $t$ and $t+1$.
We state this canonical model because it sets up our modeling framework, which is fundamentally different than the standard paradigm of stochastic programming (for multistage problems). However, it sets the foundation for searching over policies which is fundamental to our approach. We refer interested readers to \cite{ADP_Powell} for more detailed explanation of these elements of sequential stochastic decision problems.

In the rest of this section, we describe the major classes of policies that we can draw from to solve problem \eqnok{objective1}. There are two fundamental strategies for designing policies. The first is policy search, where we search over different classes of functions $f\in{\cal F}$ and different parameters $\theta^f \in \Theta^f$ in each class. Policy search is written as
\[
\min_{\pi =(f,\theta^f) \in ({\cal F},\Theta^f)} \E\left\{\sum_{t=0}^T C_t(S_t,X^\pi_t(S_t|\theta^f)) \big| S_0\right\}.
\]
Policies that can be identified using policy search come in two classes. The first one is policy function approximations (PFAs) including linear or nonlinear models, neural networks, and locally parametric functions. PFAs (using any of a wide range of approximation strategies) have been widely studied in the computer science literature under the umbrella of policy search (see e.g., \cite{sutton1999policy,bertsimas2012power,hadjiyiannis2011efficient,bertsimas2011hierarchy,lillicrap2015continuous,levine2014learning}.  The second one is cost function approximations (CFAs). In this approach, we use parametrically modified costs and constraints that are then minimized.  These are written as
    \bns
    X^{CFA}(S_t|\theta) = \argmin_{x_t\in{\cal X}^\pi(\theta)} {\bar C}^\pi(S_t,x_t|\theta),
    \ens
    where ${\bar C}^\pi(S_t,x_t|\theta)$ and ${\cal X}^\pi(\theta)$ are parametrically modified cost function and set of constraints.
    CFAs are widely used in industry for complex problems such as scheduling energy generation or planning supply chains, but they have not been studied formally in the research literature.
    %In special cases, PFAs and CFAs may produce optimal policies, although generally we are looking for the best within a class.

The second strategy is to construct policies based on lookahead models, where we capture the value of the downstream impact of a decision $x_t$ made while in state $S_t$.  An optimal policy can be written
\bn
X^*_t(S_t) = \argmin_{x_t \in {\cal X}_t} \left(C_t(S_t,x_t) + \E \left\{ \min_{\pi\in\Pi} \E \left\{\sum_{t'=t+1}^T C_{t'}(S_{t'},X^\pi_{t'}(S_{t'})) \big| S_{t+1}\right\} \big| S_t,x_t \right\}\right). \label{eq:optimalpolicylookahead}
\en
Equation \eqref{eq:optimalpolicylookahead} is basically Bellman's equation, but it is computable only for very special instances.  Two core strategies for approximating the lookahead portion in \eqref{eq:optimalpolicylookahead} include value function approximations (see e.g., \cite{ADP_Powell,bertsekas2011dynamic,sutton2018reinforcement,powell2004learning}) and direct lookahead approximations (see e.g., \citep{sethi1991theory,camacho2013model,birge2011introduction,donohue2006abridged}). In the former, we approximate the lookahead portion using a value function. When the lookahead problem cannot be reasonably approximated by a value function, the latter can be used to replace the model with an approximation for the purpose of approximating the future.

Policy search, whether we are using PFAs or CFAs, requires tuning parameters in our base objective function \eqref{objective1}.  By contrast, policies based on lookahead approximations depend on developing the best approximation of the future that can be handled computationally, although these still need to be evaluated using \eqref{objective1}.

An often overlooked challenge is the presence of forecasts.  These inherently require some form of lookahead, but are universally ignored when using value function approximations (the forecast would be part of the state variable).  However, stochastic lookahead approximations are typically computationally very demanding. In this paper, we are going to propose a hybrid comprised of a deterministic lookahead which is modified with parameters that have to be tuned using policy search.  This idea has been widely used in industry in an ad-hoc fashion without formal tuning of the parameters.  We develop this idea in the context of a multistage linear program using the context of a complex, time-dependent energy storage problem with time-varying forecasts.

\section{The Parametric Cost Function Approximations}
We extend the concept of policy search to include parameterized optimization problems. The CFA draws on the structural simplicity of deterministic lookahead models and myopic policies, but allows more flexibility by adding tunable parameters. This puts this methodology in the same class as parametric policy function approximations widely used in the policy search literature, with the only difference that our parameterized functions are inside an optimization problem, making them more useful for high dimensional problems.

\subsection{Basic Idea}
Since the idea of a parametric CFA is new, we begin by outlining the general strategy and then demonstrate how to apply it for our energy storage problem in Section~\ref{sec_energy}.
%We propose using parameterized optimization problems such as
%	\begin{equation} \label{optimal_policy}
%		X^{\pi}_t(S_t|\theta) = \argmin_{x_t \in \mathcal{X}^\pi(\theta)} \: \bigg\{C_t(S_t,x_t) + \sum_{f \in \mathcal{F}} \theta^c_f \phi_f (S_t, x_t)  : A_tx_t \le \bar{b}^{\pi}_t(\theta^b) \bigg\}
%	\end{equation}
%as a type of parameterized policy. Here the index $\pi$ signifies the structure of the modified set of constraints, $\theta^c$ is the vector of cost function parameters, $\theta^b$ is the vector of constraint parameters, and $\phi_f$ are the basis functions corresponding to features $f \in \mathcal{F}$.
For our work, we only consider parameterizing the constraints which can be written as
\begin{equation} \label{optimal_policy}
%	X^{\pi}_t(S_t | \theta) = \argmin_{x_t \in \mathcal{X}_t} \: \: \left( C_t(S_t, x_t) + \sum_f \theta^c_f \phi_f (S_t, x_t) \right)
X^{\pi}_t(S_t | \theta) = \argmin_{x_t \in \mathcal{X}_t} \: \: C_t(S_t, x_t), \qquad \text{s.t.} \qquad A_tx_t \le b_t + D\theta^b,
\end{equation}
%$\qquad \qquad \qquad \qquad \text{subject to} \qquad \qquad A_tx_t \le b_t + D\theta^b$,\\
%where $\theta^c_f$ is the vector of cost function parameters, $\phi_f$ are the basis functions corresponding to features $f \in \mathcal{F}$.
where $\theta^b$ is the vector of constraint parameters and $D$ is a scaling matrix. While the parametric terms can also be added to the cost function, we only consider parameterization of the constraints for our energy storage application in Section~\ref{sec_energy}.

Whether the parameterizations are in the objective function, or in the constraints, the specification of a parametric CFA parallels the specification of any statistical model (or policy).  The structure of the model is the ``art'' that draws on the knowledge and insights of the modeler.  Finding the best CFA, which involves finding the best $\theta$, is the science which draws on the power of classical search algorithms. This paper focuses on designing algorithms for finding $\theta$ for a given parameterization.

\subsection{A hybrid Lookahead-CFA policy}
There are many problems that naturally lend themselves to a lookahead policy (for example, to incorporate a forecast or to produce a plan over time), but where there is interest in making the policy more robust than a pure deterministic lookahead using point forecasts.  For this important class, we can create a hybrid policy where a deterministic lookahead has parametric modifications that have to be tuned using policy search. When parameters are applied to the constraints it is possible to incorporate easily recognizable problem structure.  For example, a supply chain management problem can handle uncertainty through buffer stocks, while an airline scheduling model might handle stochastic delays using schedule slack.  A grid operator planning energy generation in the future might schedule reserve capacity to account for uncertainty in forecasts of demand, as well as energy from wind and solar.  As with all policy search procedures, there is no guarantee that the resulting policy will be optimal unless the parameterized space of policies includes the optimal policy. However, we can find the optimal policy within the parameterized class, which may reflect operational limitations.  We note that while parametric CFAs are widely used in industry, optimizing within the parametric class is not done.

\subsection{Structure of the cost function approximation}
Assume that a lookahead policy is given as
%Consider the energy storage problem where a manager must satisfy the power demand of a building. The manager has a stochastic supply of renewable energy at no cost, an unlimited supply from the power grid at a stochastic price, and access to a local rechargeable storage device. Every period the manager must determine what combination of energy sources to use to satisfy the power demand, how much energy to store, and how much to sell back to the grid. Given the manager has access to point forecasts of future exogenous information he or she can use the following lookahead policy to determine how to allocate their energy,	
\begin{equation}
	X^{\text{D-LA}}_t(S_t) = \argmin_{x_t, (\tilde{x}_{t,t'}, t'=t+1, \ldots, \bar T)} \: \: \left( c_t x_{t} + \:  \sum^{\bar T}_{t' = t+1} \tilde{c}_{t,t'} \tilde{x}_{t,t'} \right), \label{energy_LA}
\end{equation}
where $\bar T = \min(t+H, T)$, $\tilde{S}_{t,t'+1} = \tilde{S}^M(\tilde{S}_{t,t'}, \tilde{x}_{t,t'}, \tilde W_{t,t'+1})$, and $H$ is the size of the lookahead horizon. If the cost function, transition function, and constraints of $X^{\pi}(\cdot)$ are all linear, this policy can be expressed as the following linear program\\
%\begin{equation}\label{PCFA}
\[
X^{\text{D-LA}}_t(S_t) = \argmin_{x_t, \tilde{x}_{t,t'}, t'=t+1, \bar T} \left\{c_t x_t + \tilde{c}_t \tilde{x}_t
		: A_t x_t \leq b_t, \: \tilde{A}_{t} \tilde x_t \leq \tilde{b}_t, \:\: x_t, \tilde x_t \geq 0  \right\},
\]
%\end{equation}
where $\tilde{a}_t = \{\tilde{a}_{t,t'} \: : \: t'=t+1,..., \bar T\}$ for any vector/matrix $\tilde{a}_{t,t'}$.
%$\tilde{c}_t = \{\tilde{c}_{t,t'} \: : \: t'=t+1,..., t+H\}$, $\tilde{A}_t = \{\tilde{A}_{t,t'} \: : \: t'=t+1,..., t+H\}$, and $\tilde{b}_t = \{\tilde{b}_{t,t'} \: : \: t'=t+1,..., t+H\}$.
Parametric terms can be appended to the cost function and existing constraints, and new ones can also be added to the existing model. Often the problem setting will influence how the policy should be parameterized. In particular, there are different ways to parameterize the above-mentioned policy including parameterizing the cost vector, the A-matrix, and the right hand side vector. In particular, assuming that all the uncertainty is restricted to the right hand side constraints, we can parameterize the vector $b_t$ such that the parameterized policy becomes\\
\[
X^{\text{LA-CFA}}_t(S_t|\theta) =  \argmin_{x_t, \tilde{x}_{t,t'}, t'=t+1, \bar T} \left\{c_t x_t + \tilde{c}_t \tilde{x}_t
		: A_t x_t \leq b_t(\theta), \: \tilde{A}_{t} \tilde x_t \leq \tilde{b}_t(\theta), \:\: x_t,\tilde x_t \geq 0\right\},
\]
where $\theta$ is a vector of tunable parameters. Parametric modifications can be designed specifically to capture the structure of a particular policy. For storage problems, the idea of using buffers and inventory constraints to manage storage is intuitive and easily incorporated into a deterministic lookahead. In particular, a lower buffer guarantees the decision maker will always have access to some stored quantities. Conversely, an upper threshold will make sure some storage space remains available for unexpected orders. Hence, representing the approximated future storage level at time $t'$ given the information available at time $t$ by $R_{t,t'}$, we can have $\theta^L \leq R_{t,t'} \leq \theta^U \text{ for } t' > t$. Although it can greatly increase the parameter space, the upper and lower bounds can also depend on $t'-t$, as in
\begin{equation}\label{TimeIN}
	\theta^L_{t'-t} \leq R_{t,t'} \leq \theta^U_{t'-t} \text{ for } t' > t.
\end{equation}
The resulting modified deterministic problem is no harder to solve than the original deterministic problem (where $\theta^L = 0$ and $\theta^U = R^{max}$). We now have to use stochastic search techniques to solve the policy search problem to optimize $\theta$.

There are also different policies for parameterizing the right hand side adjustment. A simple form is a lookup table indexed by $t'-t$ as in equation (\ref{TimeIN}). Although it may be simple, a lookup table model for $\theta$ means that the dimensionality increases with the horizon which can complicate the policy search process.
%In the energy storage example, $f^E_{t,t'}$ represents the forecast of the amount of renewable energy available at time $t'$ given the information available at time $t$. A policy maker may use the parameterization $\theta_{t'-t} \cdot f^E_{t,t'}$ to intentionally overestimate or underestimate the amount of future renewable energy. The policy maker may set $\theta_{t'-t} \leq 1$ to have more conservative policy and avoid the risk of running out of energy.
This type of parameterization is not limited to just modifying the point estimate of exogenous information. If the modeler has sufficient information such as the cumulative distribution function, one can even exchange the point estimate with the quantile function. The lookup table in time parameterization is best if the relationship between parameters in different periods is unknown.

Instead of having an adjustment $\theta_\tau = \theta_{t'-t}$ for each time $t+\tau$ in the future, one can use instead a parametric function of $\tau$, which reduces the number of parameters that we have to estimate. For example, we might use the parametric adjustment given by
\[
	\theta_1^L\cdot e^{\tau \theta_2^L } \leq R_{t,t'} \leq \theta_1^U \cdot e^{\tau \theta_2^U} \text{ for } t' > t.% \text{ and constants } \alpha, \beta \in \mathbb{R}.
\]
These parametric functions of time can also be used to directly modify the approximated future in the lookahead model.
%For example, in the energy storage example, the policy maker may use the parameterization $f^E_{t,t'} \cdot \theta_1 e^{\theta_2\cdot (t'-t)}$ to replace the forecasted amount of future renewable energy, $f^E_{t,t'}$.

We should point out that the parameterization scheme can be more general to better capture the underlying uncertainty in the model. In particular, one can also use the parameterized forms of $b(\theta(\sigma))$, $c(\theta(\sigma))$, $A(\theta(\sigma))$, where $\sigma$ represents an estimated variance of noise corresponding to the all sources of uncertainty. We let $\sigma = 0$ correspond to the case of having perfect information about the future over the horizon, which is the same as assuming that the base model is deterministic. In this case, the parameterized model should satisfy $b_t(\theta(0))=b_t$, $c_t(\theta(0))=c_t$, $A_t(\theta(0))=A_t$. An important example of parameterization would be affine with the general form of $b_t(\theta) = \theta^b_0 + \theta^b_1 b_t$ (with similar modifications for the elements of the matrix A and the cost vector c). If $\sigma = 0$, we would use $\theta^b_0 = 0$ and $\theta^b_1 = 1$. Finding the proper choice of parameterization is an art of modeling the problem which represents its own research challenges within the CFA approach and is beyond the scope of this paper.

\section{Optimizing the parameters of the CFA model}
%Policy Function Approximations (PFAs) and Cost Function Approximations (CFAs) are structurally different in the sense that one uses analytic functions and the other uses parameterized optimization problems, respectively, to make decisions. However, they are both subclasses of the same general class of parameterized policies, $X^{\pi}_t : \mathcal{S}_t \times \Theta \rightarrow \mathcal{X}_t$, and their optimal parameterization, $\theta^*$, can be found by solving
To tune our parameterized policy in the CFA model, we need to solve
\begin{equation}\label{cum_reward}
	\min_{\theta \in \Theta} \: \left\{F(\theta):= \mathbb{E}_\omega\left[\bar{F}(\theta, \omega) \right]= \mathbb{E}\left[ \sum_{t=0}^T C_t \bigg(S_{t}(\omega),X^\pi_{t}(S_{t}(\omega),\theta) \bigg) \bigg|S_0 \right]\right\},
%\mathbb{E}\left[ \sum_{t=0}^T C(S_t,X^\pi_t(S_t|\theta)) \: \big| \: S_0 \right],
\end{equation}
where
$S_{t+1}(\omega) = S^M(S_t(\omega), X^{\pi}_t(S_t(\omega),\theta), W_{t+1}(\omega))$ for every $\omega \in \Omega$. If $F(\theta)$ is well-defined, finite-valued, convex, and continuous at every $\theta$ in the nonempty, closed, and convex set $\Theta \subset \mathbb{R}^n$, then an optimal $\theta^* \in \Theta$ exists and can be found by stochastic approximation (SA) algorithms. However, when $F(\theta)$ is possibly nonconvex, SA-type methods can be modified to find stationary points of the above problem. {\color{black} In Subsection~\ref{CFA_model1}, we propose the classical SA algorithm and discuss its convergence properties when a noisy gradient of the objective function in \eqnok{cum_reward} is available. We then provide a randomized SA algorithm in Subsection~\ref{CFA_model1} which only requires noisy observations of the objective function.

\subsection{Stochastic Gradient method for optimizing the CFA model}\label{CFA_model1}
Our goal in this subsection is to solve problem \eqnok{cum_reward} under specific assumptions on $F(\theta)$. Stochastic approximation algorithms require computing stochastic (sub)gradients of the objective function iteratively. Due to the special structure of $F(\theta)$, its (sub)gradient can be computed recursively under certain conditions as shown in the next result.}

\begin{proposition}\label{prop_grad}
	Assume $\bar{F}(\cdot, \omega)$ is convex/concave for every $\omega \in \Omega$, $\theta$ is an interior point of $\Theta$, and $F(\cdot)$ is finite valued in the neighborhood of $\theta$. If distribution of $\omega$ is independent of $\theta$, we have
	\begin{equation*}
		\nabla_{\theta}F(\theta) = \mathbb{E}[\nabla_{\theta}\bar{F}(\theta, \omega)],
	\end{equation*}
	where
	\begin{equation}\label{eq:stoch_gradient}
		\nabla_{\theta}\bar{F}(\theta) = \bigg(\frac{\partial C_0}{\partial X^\pi_0} \cdot \frac{\partial X^\pi_0}{\partial \theta} \bigg) + \sum^T_{t=1} \bigg[ \bigg( \frac{\partial C_{t}}{\partial S_{t}} \cdot \frac{\partial S_{t}}{\partial \theta} \bigg) + \bigg( \frac{\partial C_{t}}{\partial X^\pi_{t}} \cdot \bigg( \frac{\partial X^\pi_{t}}{\partial S_{t}} \cdot \frac{\partial S_{t}}{\partial \theta} + \frac{\partial X^\pi_{t}}{\partial \theta} \bigg) \bigg) \bigg], \\
        \end{equation}
        \begin{equation*}
\text{and} \qquad  	\frac{\partial S_{t}}{\partial \theta} = \frac{\partial S_{t}}{\partial S_{t-1}} \cdot \frac{\partial S_{t-1}}{\partial \theta} + \frac{\partial S_{t}}{\partial X^\pi_{t-1}} \cdot \bigg[ \frac{\partial X^\pi_{t-1}}{\partial S_{t-1}} \cdot \frac{\partial S_{t-1}}{\partial \theta} + \frac{\partial X^\pi_{t-1}}{\partial \theta} \bigg],
        \end{equation*}
in which the $\omega$ is dropped for simplicity.
\end{proposition}

\begin{proof}
If $\bar{F}(\cdot, \omega)$ is convex or concave for every $\omega \in \Omega$, $\theta$ is an interior point of $\Theta$, and $F(\cdot)$ is finite valued in the neighborhood of $\theta$, then we have $\nabla_{\theta} \left(\mathbb{E} \left[\bar{F}(\theta,\omega)\right]\right) = \mathbb{E} \left[\nabla_{\theta} \bar{F}(\theta,\omega)\right]$ by \cite{Stra65}. Applying the chain rule, we find
	\begin{equation*}
	\begin{split}
	 \nabla_{\theta}\bar{F}(\theta) &= \frac{d}{d \theta} \bigg[ C_0(S_0, X^\pi_0) + \sum^T_{t=1} C(S_{t},X^\pi_{t}) \bigg]
                        		%&= \frac{\partial}{\partial \theta} C_0(S_0, X_0(S_0|\theta)) + \frac{\partial}{\partial \theta} \bigg[  \sum^T_{t'=1} C(S_{t'},X_{t'}(S_{t'}|\theta)) \bigg] \\
                        		= \bigg(\frac{\partial C_0}{d X^\pi_0} \cdot \frac{d X^\pi_0}{\partial \theta} \bigg) + \bigg[  \sum^T_{t=1} \frac{d}{d \theta} \: C(S_{t},X^\pi_{t}) \bigg] \\
                        		&= \bigg(\frac{\partial C_0}{d X^\pi_0} \cdot \frac{d X^\pi_0}{\partial \theta} \bigg) + \sum^T_{t=1} \bigg[ \bigg( \frac{\partial C_{t}}{\partial S_{t}} \cdot \frac{\partial S_{t}}{\partial \theta} \bigg) + \bigg( \frac{\partial C_{t}}{\partial X^\pi_{t}} \cdot \frac{d X^\pi_{t}}{d \theta} \bigg) \bigg] \\
                        		&= \bigg(\frac{\partial C_0}{d X^\pi_0} \cdot \frac{d X^\pi_0}{\partial \theta} \bigg) + \sum^T_{t=1} \bigg[ \bigg( \frac{\partial C_{t}}{\partial S_{t}} \cdot \frac{\partial S_{t}}{\partial \theta} \bigg) + \bigg( \frac{\partial C_{t}}{\partial X^\pi_{t}} \cdot \bigg( \frac{\partial X^\pi_{t}}{\partial S_{t}} \cdot \frac{\partial S_{t}}{\partial \theta} + \frac{\partial X^\pi_{t}}{\partial \theta} \bigg) \bigg) \bigg], \\
                        \end{split}
                        \end{equation*}
where $\tfrac{\partial S_{t}}{\partial \theta} = \tfrac{\partial S_{t}}{\partial S_{t-1}} \cdot \tfrac{\partial S_{t-1}}{\partial \theta} + \tfrac{\partial S_{t}}{\partial X^\pi_{t-1}} \big[ \tfrac{\partial X^\pi_{t-1}}{\partial S_{t-1}} \cdot \tfrac{\partial S_{t-1}}{\partial \theta} + \tfrac{\partial X^\pi_{t-1}}{\partial \theta} \big]$.
\end{proof}
Note that if $\bar F(\theta)$ is not differentiable, then its subgradient can be still computed using \eqnok{eq:stoch_gradient}. However, when $\bar F(\theta)$ is not convex (concave), its subgradient may not exist and the concept of generalized subgradient should be employed.

If $\nabla_{\theta}\bar{F}(\theta,\omega)$ exists for every $\omega \in \Omega$, the ability to calculate its unbiased estimator allows us to use stochastic approximation techniques to determine the optimal parameters, $\theta^*$, of the CFA policy model. Below is an iterative SA algorithm for optimizing the CFA model.
   \begin{algorithm}[H]
   \caption{The stochastic gradient method for the CFA model (SG-CFA)}
   \label{CFA_Grad}
   \begin{algorithmic}[1]
    	\STATE Input: $\theta^0$, $N$.\\

{\bf For $n=1,2,\ldots, N$:}

\vgap

{\addtolength{\leftskip}{0.2in}
						\STATE Generate a trajectory $\omega^{n}$ where
							\[
S^n_{t+1}(\omega^n) = S^M(S^n_t(\omega^n), X_t^\pi(S^n_t(\omega^n),\theta^{n-1}),W_{t+1}(\omega^n)) \qquad t=0,1,\ldots,T-1,
\]
						\STATE Compute the gradient estimator using equation \eqref{eq:stoch_gradient}. % $\nabla_{\theta}\bar{F}(\theta^{n-1},\omega^n)$							
							
						\STATE Update policy parameters as
							\begin{equation}
								\theta^n = \theta^{n-1} - \alpha_{n-1}\nabla_{\theta}\bar{F}(\theta^{n-1},\omega^n)|_{\theta = \theta^{n-1}}
							\end{equation}
}
{\bf End For}
    \end{algorithmic}
    \end{algorithm}

For the convergence of the above algorithm, we need to assume  the following conditions (\cite{robbins1951stochastic}):
\begin{itemize}
\item [a)] The stochastic subgradient $g^n$ computed at the $n$-th iteration satisfies\\
$\mathbb{E}\bigg[g^{n+1} ({\theta}^{n} - \theta^*) \bigg| \mathcal{F}^n\bigg] \le 0$, and $\|g^n\| \leq B_g$ a.s.,
\item [b)] For any $\theta$ where $|\theta - \theta^*| > \delta> 0$, there exists $\epsilon > 0$ such that $\|\mathbb{E}[g^{n+1}|\mathcal{F}^n]\| > \epsilon$.
\end{itemize}
Furthermore, we need to assume that stepsizes $\alpha_n$ satisfy  %\eqref{eq:SG_con1} - \eqref{eq:SG_con2}.
	{\begin{equation}\label{eq:SG_con1}
		\alpha_n > 0 \: \: \text{a.s.}, \qquad \qquad  \sum^{\infty}_{n=0} \alpha_n = \infty \: \: \text{a.s.}, \qquad \qquad
		\E \left[\sum^{\infty}_{n=0} \alpha^2_n \right]< \infty.
	\end{equation}

%	\begin{equation*}
%		\alpha_n > 0, \: \: \text{a.s}.
%	\end{equation*}
%	\begin{equation*}
%		\sum^{\infty}_{n=0} \alpha_n = \infty, \: \: \text{a.s}.
%	\end{equation*}
%	\begin{equation*}
%		\mathbb{E} \left[ \sum^{\infty}_{n=0} (\alpha_n)^2 \right] < \infty.
%	\end{equation*}
If the above conditions hold, $F(\cdot)$ is continuous and finite valued in the neighborhood of every $\theta$, in the nonempty, closed, bounded, and convex set $\Theta \subset \mathbb{R}^d$ such that $\bar{F}(\cdot,\omega)$ is convex for every $\omega \in \Omega$ where $\theta$ is an interior point of $\Theta$, then $\lim_{n \rightarrow \infty} \: \theta^n \longrightarrow \theta^* \: \: \text{a.s}$. Although any stepsize rule that satisfies the previous conditions will guarantee asymptotic convergence, we prefer parameterized rules that can be tuned for quicker convergence rates. Therefore, we limit our evaluation of the algorithm to how well it does within $N$ iterations. The CFA Algorithm can be described as a policy, $\theta^{\pi}(S^n)$, with a state variable, $S^n = \theta^n$ plus any parameters needed to compute the stepsize policy, and where $\pi$ describes the structure of the stepsize rule. If $\theta^{\pi,n}$ is the estimate of $\theta$ using stepsize rule $\pi$ after $n$ iterations, then our goal is to find the rule that produces the best performance (in expectation) after we have exhausted our budget of $N$ iterations.  Thus, we wish to solve $\min_{\pi} \mathbb{E} \left[\bar{F}(\theta^{\pi,N},\omega)\right]$ i.e., finding the best stepsize rule that minimizes the terminal cost within $N$ iterations. For our numerical experiments, we use two well-known stepsize policies, namely, the adaptive gradient algorithm (AdaGrad) \cite{duchi2011adaptive} and the Root Mean Square Propagation (RMSProp) \cite{TieHin12}, shown to perform well in practice. In particular, the AdaGrad modifies the individual stepsize for each coordinate of the updated parameter, $\theta$, based on previously observed stochastic gradients using
	\[
		\alpha_{n,j} =  \frac{\eta}{\sqrt{G_{n}(j,j) + \epsilon}}, \ \  j=1,\ldots,d,
	\]
	where $\eta>0$ is a scalar learning rate, $G_n \in \mathbb{R}^{d \times d}$ is a diagonal matrix where each diagonal element is the sum of the squares of the stochastic gradients with respect to $\theta$ up to the current iteration $n$, and $\epsilon>0$ avoids division by zero. On the other hand, the RMSProp uses a scaler stepsize based on a running average of previously observed stochastic gradients to scale the current stepsize as
\[
\alpha_n =  \frac{\eta}{\sqrt{\bar g^{n}}}, \qquad \bar g_{n}=\beta \bar g^{n-1}+(1-\beta)\|g^n\|^2,
\]
where $\eta$ is the learning rate and $\beta \in (0,1)$ is the running weight.
% and $g_n$ is the observed stochastic gradient at time $n$.
%For our simulations we set $\eta = 0.1$. This method is applied to a numerical example in Section~\ref{sec_energy}.

Note that $F(\theta)$ can be generally nonsmooth and nonconvex and hence, its subgradient does not exist everywhere. While one can define generalized subgradients for this function, one can also define a smooth approximation of $F(\theta)$ and then try to apply a stochastic approximation algorithm to minimize this function. {\color{black}We pursue this idea in the next subsection.

\subsection{Stochastic Gradient-free method for optimizing the CFA model}\label{CFA_model2}
As mentioned in the previous subsection, the gradient of $F(\theta)$ in \eqnok{eq:stoch_gradient} is only computable under restricted conditions. However, noisy values of $F(\theta)$ can be obtained through simulation. This motivates using techniques from simulation-based optimization where even the shape of the function may not be known (see e.g., \cite{fu2015handbook} and the references therein). In this subsection, we provide a zeroth-order SA algorithm and establish its finite-time convergence analysis to solve problem \eqnok{cum_reward}. For simplicity, we allow $\theta$ to take arbitrary values i.e, $\Theta =\bbr^d$ throughout this subsection.}

A smooth approximation of the function $F(\theta)$ can be defined as the following convolution
\beq \label{rand_smooth_func}
F_{\eta}(\theta) = \frac{1}{(2 \pi)^{\frac{d}{2}}} \int F(\theta+\eta v) e^{-\frac{1}{2}\|v\|^2} \,dv =\bbe_v[F(\theta+\eta v)].
\eeq
where $\eta>0$ is the smoothing parameter and $v \in \bbr^d$ is a Gaussian random vector whose mean is zero and covariance is the identity matrix. \cite{NesSpo17} provide the following result about the properties of $F_{\eta}(\cdot)$.

\begin{lemma} \label{smth_approx}
The following hold for any Lipschitz continuous function $F$ with constant $L_0$.
\begin{itemize}
\item [a)] The function $F_\eta$ is differentiable and its gradient is given by
\beq \label{smth_approx_grad}
\nabla F_{\eta}(\theta) = \frac{1}{(2 \pi)^{\frac{d}{2}}} \int \tfrac{F(\theta+\eta v)-F(\theta)}{\eta} v e^{-\frac{1}{2}\|v\|^2} \,dv,
\eeq
\item [b)] The gradient of $F_\eta$ is Lipschitz continuous with constant $L_{\eta} = \frac{\sqrt{d}}{\eta}L_0$, and for any $\theta \in \bbr^d$, we have
\beqa \label{rand_smth_close}
|F_{\eta}(\theta)-F(\theta)| &\le& \eta L_0 \sqrt{d},\\\
\bbe_v [\|[F(\theta+\eta v)-F(\theta)]v\|^2] &\le& \eta^2 L_0^2(d+4)^2.\label{bnd_grad}
\eeqa
%\item [c)]
%For any $x \in \bbr^n$,
%\beq \label{stoch_smth_approx_grad}
%\frac{1}{\mu^2}\bbe_u[\{f(x+\mu u)-f(x)\}^2\|u\|^2] \le \frac{ \mu^2}{2}L^2(n+6)^3 + 2(n+4)\|\nabla f(x)\|^2.
%\eeq
\end{itemize}

\end{lemma}
{\color{black}We now present an SA-type algorithm only using noisy values of $F$ to solve problem \eqnok{cum_reward}.}

\begin{algorithm}[!htbp]
   \caption{The stochastic gradient-free method for the CFA model (SGF-CFA)}
   \label{CFA_Grad_F}
   \begin{algorithmic}[1]
    	\STATE Input: $\theta^0 \in \bbr^d$, an iteration limit $N$, a positive sequence convergent to zero $\{\eta_k\}_{k \ge 1}$, and a probability mass function (PMF) $P_R(\cdot)$ supported on $\{1,\dots,N\}$ {\color{black}given by

\beq \label{def_PR}
P_R(R=k) = \frac{\alpha_k}{\sum_{k'=1}^N \alpha_{k'}} \qquad k =1,\ldots,N.
\eeq
}		

{\bf For $k=1,\ldots, N$:}

%\vgap

{\addtolength{\leftskip}{0.2in}

		\STATE Generate a trajectory $\omega^{k}$ where $S^k_{t+1}(\omega^k) = S^M(S^k_t(\omega^k), X_t^\pi(S^k_t(\omega^k)|\theta^{k-1}),W_{t+1}(\omega^k))$, and a random Gaussian vector $v_k$ to compute the gradient estimator as
 \beq \label{grad_zorth}
 G_{\eta_k}(\theta^{k-1},\omega^k) = \tfrac{\bar{F}(\theta^{k-1}+\eta_k v_k,\omega^k)-\bar{F}(\theta^{k-1},\omega^k)}{\eta_k}v_k
 \eeq							
							
	   \STATE Update policy parameters
							\begin{equation} \label{update_theta}
								\theta^k = \theta^{k-1} - \alpha_k G_{\eta_k}(\theta^{k-1},\omega^k).
							\end{equation}
}

{\bf End For}
%\vgap

		\STATE Generate a random index $R$ according to $P_R$ and output $\theta^R$.
    \end{algorithmic}
    \end{algorithm}
Note that the quantity defined in \eqnok{grad_zorth} using only noisy evaluations of the original function $F(\theta)$ is an unbiased estimator for the gradient of the smooth approximation function $F_\eta$ i.e.,
\beq \label{unbiased_grad}
\bbe_{v,\omega} [G_\eta(\theta,\omega)] = \bbe_{v} [\tfrac{F(\theta+\eta v)-F(\theta)}{\eta} v ] = \nabla F_\eta(\theta)
\eeq
due to \eqnok{smth_approx_grad} and the fact that $v$ and $\omega$ are independent. {\color{black}Therefore, one can use this quantity and deliberately apply SA algorithms to the function $F_\eta(\theta)$ and use Lemma~\ref{smth_approx} to establish rate of convergence of these algorithms for the original function $F(\theta)$. This is the basic idea of proposing Algorithm~\ref{CFA_Grad_F} for solving problem \eqnok{cum_reward}. It should be mentioned that the framework of this algorithm has been first proposed in \cite{NesSpo17} and then widely used in the literature of stochastic optimization using zeroth-order information (see e.g., \cite{GhaLan12}). However, our choice of stepsize policy and using adaptive smoothing parameter makes Algorithm~\ref{CFA_Grad_F} different from its existing variants.

{\color{black} The algorithm uses an idea first proposed by Ghadimi and Lan (2013) of randomly choosing from the generated trajectory $\{\theta^1, \ldots, \theta^k, \ldots, \theta^N\}$, instead of using the last iterate $\theta^N$. This step is essential to establish convergence results like \eqnok{stationry_Feta2} for SA-type algorithms when applied to nonconvex stochastic optimization problems. This kind of randomization scheme seems to be the only way} for this purpose since $\min_{k=1,2,\ldots}\|\nabla F(\theta^k)\|^2$ is not computable for this class of problems. Note that the PMF of the random index of the output solution depends on the choice of stepsize policy which will be discussed later in this subsection.

Since Algorithm~\ref{CFA_Grad_F} uses an adaptive smoothing parameter $\eta_k$, its convergence analysis is slightly different than the one presented in \cite{NesSpo17} when $F$ is nonsmooth and nonconvex. Hence, we now provide its main convergence property.
}
\begin{theorem}
Let $\{\theta_k\}$ be generated by Algorithm~\ref{CFA_Grad_F} and $\bar F (\theta)$ be Lipschitz continuous with constant $L_0$. If $F(\theta)$ is bounded above by $F^*$, we have
\beq \label{stationry_Feta}
\bbe[\|\nabla F_{\eta_R} (\theta^R)\|^2] \le \frac{F^* - F(\theta^0)+L_0 \sqrt{d}\left(\eta_1+\eta_N+\sum_{k=1}^{N-1} |\eta_k - \eta_{k+1}|+ L_0^2 (d+4)^2 \sum_{k=1}^N \frac{\alpha_k^2}{2\eta_k}\right)}{\sum_{k=1}^N \alpha_k},
\eeq
where {\color{black}the expectation is taken w.r.t the randomness arising from the nature of the problem $\omega$, and the ones imposed by the algorithm, namely, Gaussian random vector $v$ and random integer number $R$ whose probability distribution} is given by \eqnok{def_PR}.
\end{theorem}
\begin{proof}
First note that $F (\theta)$ is Lipschitz continuous with constant $L$ due to the same assumption on $\bar F (\theta)$. Hence, the gradient of $F_\eta (\theta)$ is Lipschitz continuous with constant $L_\eta$ due to Lemma~\ref{smth_approx}.b which together with \eqref{update_theta} imply that
\beqa
F_{\eta_k} (\theta^k) &\ge& F_{\eta_k} (\theta^{k-1}) + \langle \nabla F_{\eta_k} (\theta^k), \theta^k - \theta^{k-1} \rangle -\frac{L_{\eta_k}}{2}\|\theta^k - \theta^{k-1}\|^2 \nn\\
&=& F_{\eta_k} (\theta^{k-1}) + \alpha_k \langle \nabla F_{\eta_k} (\theta^k), G_{\eta_k}(\theta^{k-1},\omega^k) \rangle -\frac{L_{\eta_k} \alpha_k^2}{2}\|G_{\eta_k}(\theta^{k-1},\omega^k)\|^2.\nn
\eeqa
Taking expectations of both sides, noting \eqnok{bnd_grad}, \eqnok{grad_zorth}, and \eqnok{unbiased_grad}, we obtain
\[
\bbe[F_{\eta_k} (\theta^k)] \ge F_{\eta_k} (\theta^{k-1}) + \alpha_k \|\nabla F_{\eta_k} (\theta^k)\|^2 -\frac{L_0^3 \sqrt{d}(d+4)^2 \alpha_k^2}{2\eta_k}.
\]
Summing up both sides of the above inequality and re-arranging the terms, we have
\beq \label{proof_nocvx1}
\sum_{k=1}^N \alpha_k \bbe[\|\nabla F_{\eta_k} (\theta^k)\|^2 ] \le \Delta_N+L_0^3 \sqrt{d}(d+4)^2 \sum_{k=1}^N \frac{\alpha_k^2}{2\eta_k},
\eeq
where $\Delta_N = F_{\eta_N}(\theta^N)-F_{\eta_1}(\theta^0)+\sum_{k=1}^{N-1} [F_{\eta_k}(\theta^k)-F_{\eta_{k+1}}(\theta^k)]$. Noting \eqnok{rand_smth_close}, the fact that $F (\theta) \le F^*$ for any $\theta \in \bbr^d$, \eqnok{rand_smooth_func}, and Lipschitz continuity of $F$, we have
\beqa
F_{\eta_N}(\theta^N)-F_{\eta_1}(\theta^0) &\le& F^* - F(\theta^0)+ (\eta_1+\eta_N) L_0 \sqrt{d}, \nn \\
F_{\eta_k}(\theta^k)-F_{\eta_{k+1}}(\theta^k) &=&  \bbe_v[F(\theta^k+\eta_k v)-F(\theta^k+\eta_{k+1} v)] \nn \\
 &\le& L_0 |\eta_k - \eta_{k+1}| \bbe_v[\|v\|] \le L_0 \sqrt{d} |\eta_k - \eta_{k+1}|.\nn
\eeqa
Combining \eqnok{proof_nocvx1} with the above two relations,
\eqnok{stationry_Feta} follows by noting that in the view of \eqnok{def_PR}, we have
$
\bbe[\|\nabla F_{\eta_R} (\theta^R)\|^2] = \frac{\sum_{k=1}^N \alpha_k \bbe[\|\nabla F_{\eta_k} (\theta^k)\|^2 ]}{\sum_{k=1}^N \alpha_k}.
$
\end{proof}
\begin{corollary}
Let the smoothing parameters and stepsizes of Algorithm~\ref{CFA_Grad} be given by
\beq \label{eta_alpha}
\eta_k = \frac{L_0 (d+4)}{k^\beta}, \qquad \alpha_k = \frac{1}{\sqrt{k}} \qquad k =1,\ldots, N,
\eeq
for any $\beta \in (0, \tfrac12)$. Then we have
\beq \label{stationry_Feta2}
\bbe[\|\nabla F_{\eta_R} (\theta^R)\|^2] \le \frac{1}{2\sqrt{N}} \left[F^* - F(\theta^0)+L_0^2 \sqrt{d}(d+4)\left(2+\frac{(N+1)^\beta}{\beta}\right)\right],
\eeq
where the probability distribution of $R$ is given in \eqnok{def_PR}.
\end{corollary}
\begin{proof}
Note that by \eqnok{eta_alpha}, we have
\begin{align*}
&\sum_{k=1}^{N-1} |\eta_k - \eta_{k+1}| = \sum_{k=1}^{N-1} (\eta_k - \eta_{k+1})= \eta_1-\eta_N, \\
& \sum_{k=1}^N \alpha_k \ge 2 \sqrt{N}, \qquad \sum_{k=1}^N \frac{\alpha_k^2}{\eta_k} = \frac{1}{L_0(d+4)}\sum_{k=1}^N k^{\beta-1} \le \frac{(N+1)^\beta}{\beta L_0(d+4)},
\end{align*}
which, with \eqnok{stationry_Feta}, clearly imply \eqnok{stationry_Feta2}.
\end{proof}
{\color{black}It should be mentioned that if a fixed smoothing parameter $\eta_k=\eta$ and a fixed stepsize $\alpha_k =\alpha$ are properly employed, then the upper bound would be on the order of $1/\sqrt{N}$. This rate has been obtained in \cite{NesSpo17} for the weighted average of $\bbe[\|\nabla F_{\eta}(\theta^k)\|^2]$ without introducing the random index $R$. However, to ensure convergence to the stationary point of the original problem, one should allow the smoothing parameter $\eta_k$ and stepsize $\alpha_k$ converge to $0$. Therefore, choosing these two quantities adaptively converging to $0$ would be more desirable in practice than setting them to fixed small numbers.
}
%For deterministic nonsmooth, nonconvex problems with constant stepsize and smoothing parameter, \cite{NesSpo17} obtained a similar rate of convergence for weighted average of $\bbe_v[\|\nabla F_{\eta}\|^2]$ (without introducing the random index $R$). He says that by allowing both stepsize and smoothing parameter go to zero as $k \to \infty$, he can show asymptotic convergence to the stationary point of the original problem. However, he says that the proof is ``quite long and technical'' and hence he does not include it.\\

\subsection{The linear CFA model}
Our goal in this subsection is to specialize some results from the previous subsection and provide more properties for the linear CFA model. In particular, if the objective function in \eqnok{cum_reward} is a linear function of the decisions, $x_t$, the parametric CFA policy, $X^{\pi}_t(S_t|\theta)$, which determines the decision, $x_t$, can be written as the following linear program
\beq\label{eqn:LP_main}
    		X^{\pi}_t(S_t|\theta) =x^*_t, \qquad  \tilde{x^*}_t = [x^*_t,...,\tilde{x}^*_{t,\bar T}] = \argmin_{x_t, (\tilde{x}_{t,t'}), t'=t+1, \ldots, \bar T} \: \: c_t x_t + \sum_{t'=t+1}^{\bar T} \tilde{c}_{t,t'} \tilde{x}_{t,t'},
\eeq
where $\bar T = \min(t+H,T)$, $\tilde{A}_{t}\tilde{x}_{t} \leq \tilde{b}_{t}(\theta)$ for given $\tilde{A}_{t}$ and $\tilde{b}_{t}$. The state variable, $S_t$, includes the point estimates, $(\tilde W_{t,t'})_{t'=t+1,...,\bar T}$, that are used to approximate exogenous information. If this policy is written as a linear program where the state and approximated exogenous information is only in the right hand side constraints, $\tilde{b}_t(\theta)$, then Proposition~\ref{prop_grad} for computing a stochastic subgradient of $F(\theta)$ can be simplified as follows.
\begin{proposition}
Let $\bar{F}(\theta, \omega)$ be convex in $\theta$ for every $\omega \in \Omega$, $\theta$ be an interior point of $\Theta$, and the contribution cost function be a linear function of $x$, and the transition function $S_t = S^M(S_{t-1},x_{t-1},W_{t})$ be linear in $S_{t-1}$ and $x_{t-1}$. Moreover, assume that $F(\theta)$ is finite valued in the neighborhood of $\theta$, and the policy, $X^{\pi}_t(S_t|\theta)$ is given by \eqnok{eqn:LP_main} in which  $B_t$ is the basis matrix corresponding to the basic variables for the optimal solution. Then
\beqa
\nabla_{\theta} \bar{F}(\theta, \omega) &=& \sum^T_{t=1} \bigg( \nabla_{\theta} \tilde{b}_{t}(\theta) + \nabla_{S_t} \tilde{b}_{t}(\theta) \nabla_{\theta} S_t \bigg)^T \bigg( B^{-1}_t \bigg)^T c_t, \label{grad_barF}\\
\text{where} \ \  \nabla_{\theta}S_t &=& \nabla_{S_{t-1}} S^M(S_{t-1},x_{t-1},W_t) \nabla_{\theta}S_{t-1} + \nabla_{x_{t-1}} S^M(S_{t-1},x_{t-1},W_t) \nabla_{\theta} x_{t-1}.\nonumber
\eeqa
\end{proposition}
\begin{proof}
Proof follows from the proof of Proposition~\ref{prop_grad} and the basic properties of the linear program. Hence, we skip the details.
\end{proof}
For the linear CFA model, we can also provide some properties of the objective function.
\begin{lemma}
Assume that problem \eqref{cum_reward} belongs to a class of linear programs such that
\beq\label{linear_cum_reward}
C_t\Big(S_{t}(\omega,\theta),X^\pi_{t}(S_{t},\theta)\Big) = c_t(\omega)X^\pi_{t}(S_{t},\theta),
\eeq
where $X^{\pi}_t(S_t,\theta)$ is set to \eqnok{eqn:LP_main}. Further assume that the transition function $S_t = S^M(S_{t-1},x_{t-1},W_t)$ is linear in $S_{t-1}$ and $x_{t-1}$, and $\tilde{b}_{t}$ is also linear in $\theta$ and $S_{t}$. Moreover,
%let $\theta_0 \in \bbr^d$ be the value of $\theta$ which corresponds to the case of no parameterization ($\theta_0(i)$ is usually $0$ or $1$ depending on the type of parameterization) and
let $\Theta_t(\omega)$ be the largest convex subset of $\Theta_{t-1}(\omega)$ such that the optimal basis corresponding to $\tilde{x}^*_t$ remains the same for any $\theta$ belonging to this subset with the definition of $\Theta_{-1}(\omega):=\bbr^d$. Then $F(\theta)$, defined in \eqref{cum_reward}, is linear in $\theta$ over $\bar \Theta:= \bigcap_{\omega \in \Omega} \Theta_T(\omega)$.
\end{lemma}
\begin{proof}
First note that by the definition of $\Theta_t(\omega)$, for a fixed $\omega \in \Omega$ we have $\Theta_T \subseteq \Theta_{T-1} \subseteq \cdots \Theta_0 \subseteq \bbr^d$. Noting that $S_0$ does not depend on $\theta$ and $\tilde{b}_0(\theta, S_0)$ is linear in $\theta$, we conclude that $x^*_0$ is linear in $\theta$ for any $\theta \in \Theta_0(\omega)$. This also implies that $S_1$ is linear in $\theta$. Hence, $\tilde{b}_1(\theta, S_1)$ is also linear in $\theta$ for any $\theta \in \Theta_0(\omega)$ implying that $x^*_0$ is linear in $\theta$ for any $\bigcap_{t=0,1}\Theta_t(\omega) = \Theta_1(\omega) $. The same argument holds for any $t \ge 2$. Therefore, taking the intersection of $\bigcap_{t=0,\ldots,T}\Theta_t(\omega) = \Theta_T(\omega) $ over all $\omega \in \Omega$, we conclude that $F(\theta)$ is linear for all $\theta$ belonging to this intersection set.
\end{proof}
A more general form of the above Lemma can be stated as follows.
\begin{lemma}
Consider problem \eqref{cum_reward} together with \eqnok{linear_cum_reward} and \eqnok{eqn:LP_main} such that only one of the right hand side constraints (other than inventory ones), say the $i$-th constraint, is parameterized with a $\theta$ for all periods. Further assume that the transition function $S_t = S^M(S_{t-1},x_{t-1},\omega_{t})$ is linear in $S_{t-1}$ and $x_{t-1}$, and $\tilde b^i_{t}$ is also linear in $\theta$ and $S_{t}$. Assuming that an interval $[a^l,a^u] \subset \bbr$ is given for the range of $\theta$, one can partition $[a^l,a^u]$ into subintervals such that $F(\theta)$ is a piecewise linear convex function of $\theta$ on each of these subintervals.
\end{lemma}
\begin{proof}
Assume that the linear programs in \eqnok{eqn:LP_main} are solved for all $t=0,1,\ldots,T$ with the choice of $\theta=a^u$. Moreover, assume that $[a^{l,1}_0, a^u]$ be the subinterval that the optimal basis for the linear program in \eqnok{eqn:LP_main} corresponding to $t=0$ does not change for any $\theta$ belonging to this subinterval. Hence, $X^{\pi}_0(S_0,\theta)$ is a linear function of $\theta$ on this subinterval. This also implies that the inventory constraints for the linear program of $t=1$ are also linear functions of $\theta$ on the aforementioned subinterval. Hence, there exists $a^{l,1}_1 \ge a^{l,1}_0$ such that the optimal basis for the linear program in \eqnok{eqn:LP_main} corresponding to $t=1$ does not change for any $\theta \in [a^{l,1}_1, a^u]$. Moving forward with this argument we obtain the non-decreasing sequence $\{a^{l,1}_t\}_{t=0,1,\ldots,T-1}$ such that the solution of the $t$-th linear program is a linear function of $\theta$ on the subinterval $[a^{l,1}_t, a^u]$. Consequently, the right hand side of the $T$-th linear program including the inventory constraint and the $i$-th constraint will be a linear function for all $\theta \in [a^{l,1}_{T-1}, a^u]$. Hence, the optimal solution of this linear program is a piecewise linear convex function for any $\theta \in [a^{l,1}_{T-1}, a^u]$ and so is $F(\theta)$. Repeating the above argument for the interval $[a^{l,1}_0, a^{l,1}_{T-1}]$, we obtain a new subinterval over which $F(\theta)$ is a piecewise linear convex function of $\theta$. This process can be continued until the whole interval $[a^l,a^u]$ is covered.
\end{proof}
Note that the above result can be extended when more than one constraint is linearly parameterized by $\theta$. In this case, instead of subintervals, we have subsets of the parameter space over which $F(\theta)$ is a piecewise linear convex function of $\theta$. The next result provides the optimal policy for the special case of having perfect information about the future.
\begin{lemma}\label{lem_perfect}
Assume that $H=T-1$ and we are given perfect information for all sources of uncertainty. Then the optimal policy is not to parameterize the model i.e., $\theta^* = \theta(0)$.
\end{lemma}
\begin{proof}
Given perfect forecasts for the case $H=T-1$ means that the forecast is no longer rolling over the horizon and is fixed. Hence the optimal solution of the linear program solved at $t=0$ is also optimal for all linear programs solved over the horizon.
% and factoring the forecast in any way may change this optimal solution.
%Therefore, keeping the perfect forecasts unchanged and not parameterizing it, is an optimal policy.
\end{proof}
\subsection{The Static CFA model}
In this subsection, we propose an alternative approach for solving the base model. In particular, we consider a static variant of problem \eqnok{cum_reward} given by
\beqa
	&\min_{\theta \in \Theta}& \: \left\{F^S(\theta):=\mathbb{E}\left[\bar F^S(\theta) = \sum_{t=0}^T C_t(X^\pi_t(\theta)) \right] \right\}, \label{cum_reward2}\\
&\text{where}& \qquad [X^{\pi}_0(\theta),X^{\pi}_1(\theta), \ldots, X^{\pi}_T(\theta)]^\top = \argmin_{x_0, x_1, \ldots, x_T} \: \: \sum_{t=0}^T c_t x_t, \nn\\
&\text{subject to}& \qquad  A_t x_t \leq b_t(\theta) \quad t=0,1,\ldots,T. \label{stat_CFA}
\eeqa
Indeed, to evaluate the objective function of the model for a given $\theta$, we only solve one linear program at time $t=0$. In this case to better capture forecast changes for all periods, we allow $b_t$ to be any convex function of $\theta$. We also allow each period to have its own parameterization independent of other periods. The next result provides conditions where problem \eqnok{cum_reward2} is a convex programming problem.
\begin{lemma}
Let $F^S(\theta)$ and $X^{\pi}$ be defined in \eqnok{cum_reward2} and \eqnok{stat_CFA}, respectively. Also assume that inventory constraints are not parameterized by $\theta$ and $b_t(\theta)$ is convex in $\theta$ for any $t\in \{0,1,\ldots,T\}$. Then $F^S(\theta)$ is convex in $\theta$.
\end{lemma}
\begin{proof}
Note that if inventory constraints are not parameterized by $\theta$ and $b_t(\theta)$ is convex in $\theta$ for any $t\in \{0,1,\ldots,T\}$, the feasible set of problem \eqnok{stat_CFA} is clearly convex in $(x,\theta)$. Hence, since the objective function is only linear (convex) in $x$, the optimal value of \eqnok{stat_CFA}, which is indeed $F^S(\bar \theta, \omega)$, is also convex in $\theta$. Therefore, $F^S(\theta)=\bbe[F^S(\bar \theta, \omega)]$ is convex.
\end{proof}
Note that in the above lemma, we only need the parameterization of constraints (except the inventory ones) to be convex. In this case, subgradients of $F^S(\theta)$ are available everywhere and one can use the standard stochastic approximation algorithms for convex programming such as Algorithm~\ref{CFA_Grad} to solve problem \eqnok{cum_reward2} with convergence guarantees. We should also point out that to easily compute subgradients of $\bar F^S(\theta)$, we also assume that $b_t(\theta)$ is differentiable in $\theta$. Moreover, under the aforementioned static setting, we can significantly simplify the subgradient computations since the right-hand-side of the constraints do not depend on the state variables. In particular, \eqnok{grad_barF} is reduced to
%\beq\label{grad_barF2}
\[
\nabla_{\theta} \bar{F}^S(\theta, \omega) = \sum^T_{t=0} \nabla_{\theta} b_{t}(\theta)^\top \bigg( B^{-1}_t \bigg)^\top c_t. \\
\]
%%%%%%%%%%%%%%%%%%%%%%%%%%%
%%%%%%%%%%%%%%%%%%%%%%%%%%%
\section{An Energy Storage Application}\label{sec_energy}
{\color{black}In this section, we use the setting of an energy storage application to show how we can use a parametric CFA to produce robust policies using rolling forecasts of varying quality.
%In Subsection~\ref{energy_int}, we introduce the energy storage problem and its associated optimization problem. We then provide a lookahead approximation model in Subsection~\ref{energy_param} and discuss different ways of policy parameterization to improve the performance of the approximation model.}
%Our problem is designed to test the capabilities of the algorithm, rather than representing an accurate model of a specific energy storage application.
\subsection{Problem description}\label{energy_int}
A smart grid manager must satisfy a recurring power demand with a stochastic supply of renewable energy, unlimited supply of energy from the main power grid at a stochastic price, and access to local rechargeable storage devices. This system is illustrated in Figure \ref{fig:system}.

\begin{figure}[h]
    \centering
    \includegraphics[width=0.6\textwidth]{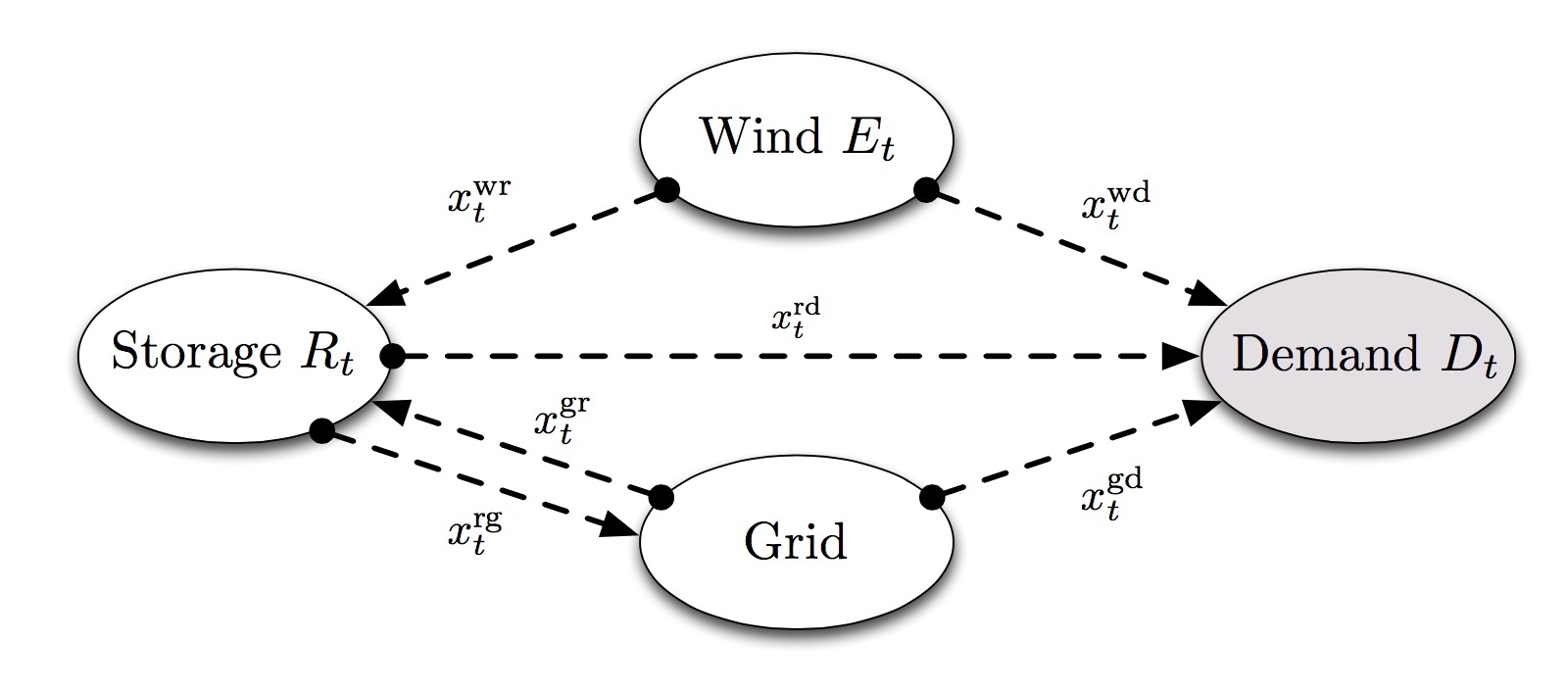}
    \caption{Energy system schematics}
    \label{fig:system}
\end{figure}

%The demand, $D_t$, has a deterministic seasonal structure
%\begin{equation}
%	D_t = \lfloor \max \{0 ,100 - 50 \sin \bigg( \frac{5 \pi t}{T} \bigg) \} \rfloor.
%\end{equation}

At the beginning of every period $t$ the manager must combine energy from different sources to satisfy the demand, $D_t$:
energy currently in storage (represented by a decision $x_t^{rd}$); newly available wind energy (represented by a decision $x_t^{wd}$); energy from the grid (represented by a decision $x_t^{gd}$).
Additionally, the manager must decide how much renewable energy to store, $x_t^{wr}$, how much energy, $x_t^{rg}$, to sell to the grid at price $P_t$, and how much energy to buy from the grid and store, $x_t^{gr}$. Hence, the manager's decision variable at $t$ is defined as the vector $x_t = (x^{wd}_t, x^{gd}_t, x^{rd}, x^{wr}_t, x^{gr}_t, x^{rg}_t)^T \geq 0$, which should satisfy the following constraints:
%\begin{equation*}%\label{eqn:con}
%		\begin{matrix}
%		x^{wd}_t & + & \beta^dx^{rd}_t & + & x^{gd}_t & \leq 		& D_t, \\
%		~ & ~ & x^{gd}_t & + & x^{gr}_t & \leq 	& G_t, \\
%		~ & ~ & x^{rd}_t & + & x^{rg}_t & \leq 	& R_t, \\
%%		\end{matrix}
%%\end{equation*}
%%\begin{equation*}%\label{eqn:con}
%%		\begin{matrix}
%		~ & ~ & x^{wr}_t & + & x^{gr}_t & \leq 	& R^{\max} - R_t, \\
%		~ & ~ & x^{wr}_t & + & x^{wd}_t & \leq 	& E_t, \\
%		~ & ~ & x^{wr}_t & + & x^{gr}_t & \leq 	& \gamma^c, \\
%		~ & ~ & x^{rd}_t & + & x^{rg}_t & \leq 	& \gamma^d, \\
%		\end{matrix}
%\end{equation*}
\begin{eqnarray}
	x^{wd}_t  + \beta^dx^{rd}_t + x^{gd}_t 	& \leq & D_t, \qquad x^{wr}_t + x^{wd}_t \leq E_t, \qquad
x^{rd}_t + x^{rg}_t \leq R_t,\label{eq:1} \\
x^{wr}_t + x^{gr}_t - x^{rd}_t - x^{rg}_t &\leq& R^{\max} - R_t, \qquad x^{wr}_t + x^{gr}_t \leq  \gamma^c, \qquad
x^{rd}_t + x^{rg}_t \leq \gamma^d, \label{eq:7}
\end{eqnarray}
where $\gamma^c$ and $\gamma^d$ are the maximum amount of energy that can be charged or discharged from the storage device. Typically, $\gamma^c$ and $\gamma^d$ are the same.

%Every hour the manager must determine what combination of energy sources to use to satisfy the power demand, how much energy to store, and how much to sell back to the grid.
The state variable at time $t$, $S_t$, includes the level of energy in storage, $R_t\in [0, R^{\max}]$ ($R^{\max}>0$ represents the storage capacity), the amount of energy available from wind and its forecast, $\{f^E_{t,t'}\}_{t' \ge t} (E_t=f^E_{t,t})$, the spot price of electricity from the grid and its forecast, $\{f^{P^g}_{t,t'}\}_{t' \ge t} (P^g_t=f^{P^g}_{t,t})$, the market price of electricity $P^d_t$, the demand $D_t$ and its forecast $\{f^D_{t,t'}\}_{t' \ge t} (D_t=f^D_{t,t})$, and the energy available from the grid $G_t$ at time $t$. Hence the state of the system can be represented by the vector $S_t = (R_t, G_t, f^E_{t,t'}, P^m_t, f^{P^g}_{t,t'}, f^D_{t,t'}) \ \ \forall t' \ge t$. In the Appendix, we describe how these forecasts can be generated.

The transition function, $S^M(\cdot)$ also explicitly describes the relationship between the state of the model at time $t$ and $t+1$ such that $S_{t+1} = S^M(S_t, x_t, W_{t+1})$, where $W_{t+1} = (E_{t+1}, P_{t+1}, D_{t+1})$ is the exogenous information revealed at $t+1$. In our numerical experiments, we assumed that $W_{t+1}$ is independent of $S_t$, but the CFA algorithm can work with any sample path provided by an exogenous source, which means it can handle an exogenous information process where $W_{t+1}$ may depend on $S_t$ and/or $x_t$. Indeed, the CFA method belongs to the class of {\it data driven} algorithms, where we do not need a model of the exogenous process. The relationship of storage levels between periods is defined as:
	\begin{equation}\label{eqn:trans}
		R_{t+1} = R_{t} - x^{rd}_t + \beta^c x^{wr}_t + \beta^c x^{gr}_t - x^{rg}_t,
	\end{equation}
where $\beta^c \in (0, 1)$ and $\beta^d \in (0, 1)$, are the charge and discharge efficiencies. Denoting the penalty of not satisfying the demand by $C^P$, for given state $S_t$ and decision $x_t$, the profit realized at $t$ is given by
\begin{equation}\label{contr_func}
	C_t(S_t, x_t) = C^P( D_t - x_t^{wd} - \beta^d x_t^{rd} - x_t^{gd}) - P^m_t (x_t^{wd} + \beta^d x_t^{rd} + x_t^{gd}) - P^g_t (\beta^d x^{rg}_t - x^{gr}_t - x^{gd}_t).
\end{equation}
%where $C^P$ is the penalty of not satisfying the demand.
%and $C(S_t,x_t)$ is the profit realized at $t$ given the current state is $S_t$ and the decision is $x_t$.
%Our goal is to find the policy $\pi$ that solves
%{\color{black}\beq \label{objective}
%\min_{\theta} \: \: \mathbb{E} \bigg[ \sum^T_{t=0} C_t(S_t, X^{\pi}_t(S_t|\theta)) \: \bigg| \: S_0  \bigg] \quad    \text{where} \quad   S_{t+1} = S^M(S_t, X^{\pi}_t(S_t|\theta), W_{t+1}) \ \  \forall t \in [0,T],
%\eeq
%and $X^{\pi}_t(S_t|\theta)$ is given by \eqnok{optimal_policy} in which $\theta^c_f=0$.}

\subsection{Policy Parameterizations}\label{energy_param}
%If the contribution function, transition function and constraints are linear, a deterministic lookahead policy can be constructed as a linear program if point forecasts of exogenous information are provided.
For our deterministic lookahead, {\color{black} by noting \eqnok{contr_func}, we solve subproblem \eqref{energy_LA} subject to constraints \eqref{eq:1} - \eqref{eqn:trans} for $t' \in [t+1,t+H]$.
We call this deterministic lookahead policy the benchmark policy, and use it to estimate the degree to which the parameterized policies are able to improve the results in the presence of uncertainty. %We call this deterministic lookahead policy the \emph{benchmark policy}, test it on the previously described problem, and compare it's performance to the following parameterizations of constraints \eqref{eqn:b_con}.
There are different ways of parameterizing the policy in this lookahead model. A few examples are as follows.}

\begin{itemize}
%	\item {\bf Capacity Constraints}: This parameterization limits the amount of energy in storage and guarantees there is capacity to purchase inexpensive energy. An upper bound constraint is easily created by multiplying the capacity of the storage device, $R^{\max}$ by the parameter $\theta_{t'-t}$. This changes \eqref {eq:4} to
%        \begin{equation}\label{eqn:UB}
%        		x^{wr}_{t,t'} + x^{gr}_{t,t'} \leq R^{\max} \cdot \theta^U_{t'-t} - f^R_{t,t'}
%        \end{equation}
%    	where $\theta_{t'} \in [0, 1]$ and $t' \in [t,t+H]$. Parameterized lower constraints are incorporated into the policy by creating the additional linear constraints
%        \begin{equation}\label{eqn:LB}
%        		- x_t^{rd} - x_t^{rg} + R_t \geq R^{\max} \cdot \theta^L_{t'-t}
%        \end{equation}
%        where $\theta^L_{t'} \in [0, 1]$ and $t' \in [t+1,t+H]$.
\item {\bf Constant forecast parameterization} - This parameterization uses a single scalar to modify the forecast amount of renewable energy for the entire horizon. Hence, the second constraint in \eqref{eq:1} is changed to
        		\begin{equation}\label{eq:}
        			x^{wr}_{t,t'} + x^{wd}_{t,t'} \leq \theta \cdot f^E_{t,t'}.
        		\end{equation}

	\item {\bf Lookup table forecast parameterization} - Overestimating or underestimating forecasts of renewable energy influences how aggressively a policy will store energy. We modify the forecast of renewable energy for each period of the lookahead model with a unique parameter $\theta_{\tau}$. This parameterization is a lookup table representation because there is a different $\theta$ for each lookahead period, $\tau = 0, 1, 2, ...$. This changes \eqref {eq:} to
        \begin{equation}\label{eqn:UFP}
        		x^{wr}_{t,t'} + x^{wd}_{t,t'} \leq \theta_{t' - t} \cdot f^E_{t,t'}.
        \end{equation}
         where $t' \in [t+1,t+H]$ and $\tau = t' -t$. If $\theta_{\tau} < 1$ the policy will be more conservative and decrease the risk of running out of energy. Conversely, if $\theta_{\tau} > 1$ the policy will be more aggressive and less adamant about maintaining large energy reserves.
	
	%\item {\bf Exponential Function} - Instead of calculating a set of parameters for every period within the lookahead model we make our parameterization a function of time and two parameters. The policy constraints \eqref{eq:5} are then changed to
%	\begin{equation}\label{eqn:ETD}
%		x^{wr}_{t,t'} + x^{wd}_{t,t'} \leq f^E_{t,t'} \cdot \theta_1 \cdot e^{\theta_2 \cdot (t'-t)}.
%	\end{equation}
\end{itemize}

\subsection{Numerical Experiments}
In this subsection, we test the aforementioned parameterizations of the deterministic lookahead policy defined on variations of the energy storage problem by providing the same forecasts of exogenous information for the benchmark and parameterized policies.
%Our goal is to show that parameterizing the benchmark policy within the CFA framework and optimizing parameter values can improve the performance over that achieved by the benchmark policy.
We say the parameterization $\theta$ outperforms the nonparametric benchmark policy if it has positive \emph{policy improvement} $\Delta F^{\pi}(\theta)$, given by
%We define the policy improvement, $\Delta F^{\pi}(\theta)$, of parameterization $\pi(\theta)$ as
	\begin{equation}\label{improve}
		\Delta F^{\pi}(\theta) = \frac{F^{\pi}(\theta) - F^{\text{D-LA}}}{|F^{\text{D-LA}}|},
	\end{equation}
where $F^{\pi}(\theta)$ is the average profit (negative cost defined in \eqnok{contr_func}) generated by $\theta$ and $F^{\text{D-LA}}$ is the average profit generated by the unparameterized deterministic lookahead policy described by equation \eqref{energy_LA}. In all of our experiments, we compute these averaged profits over a testing data set including $1000$ random samples.

In our first set of experiments, we examine the performance of the lookup table parameterization policy with $H=23$ under perfect forecasts ($\sigma_{E}=0$) and noisy one ($\sigma^2_{E}=40$). In particular, we first set all values of $\theta$ to $1$ and then do a one-dimensional search over each coordinate of $\theta$. As it can be seen from Figure~\ref{fig:lookup_param_perfect}, under perfect forecasts, the optimal value for each coordinate of $\theta$ is $1$ while the others are set to $1$ as suggested by Lemma~\ref{lem_perfect}. However, when noisy forecasts are used, the optimal values for $\theta$ may be different than $1$.
%It is also worth noting that the last few coordinates of $\theta$ have almost no effect on the performance of a given policy. Intuitively, the possible reason would be the fact that a given forecast for distant future has small affect on the present decision.
\begin{figure}
        \centering
%        \begin{subfigure}{.5\textwidth}
%          \centering
         \includegraphics[width=0.4\linewidth]{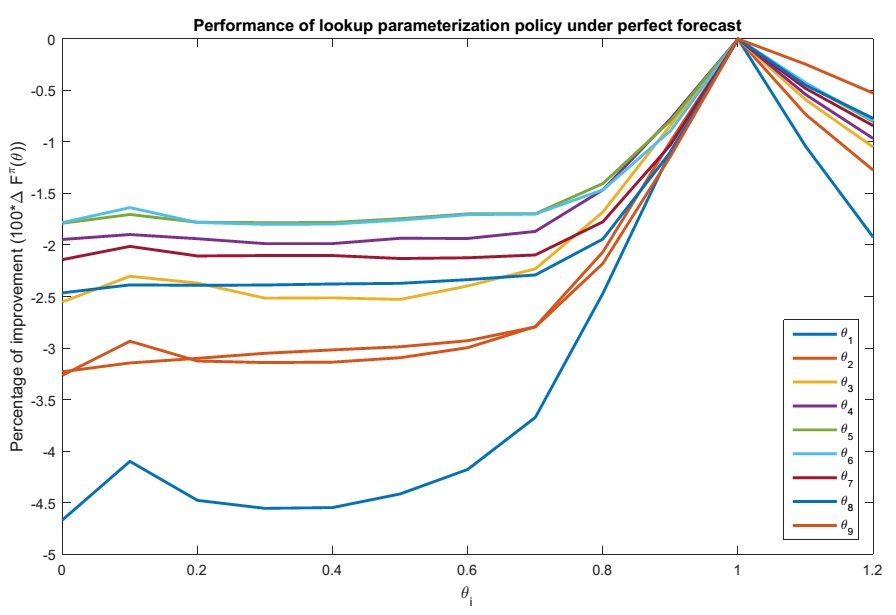}
         \label{fig:lookup_perfect}
%        \end{subfigure}%
%        \begin{subfigure}{.5\textwidth}
%          \centering
          \includegraphics[width=0.4\linewidth]{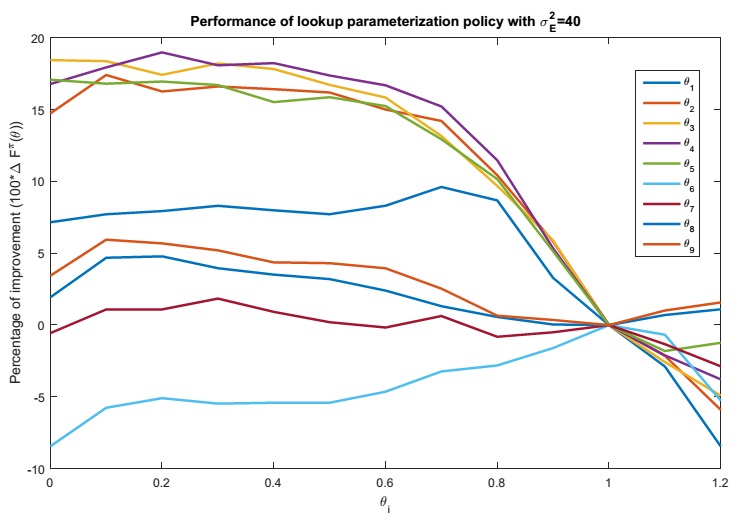}
          \label{fig:lookup_perfect2}
%        \end{subfigure}
\caption{Averaged performance of lookup parameterization policy under perfect forecasts (right) and noisy one ($\sigma^2_{E}=40$). Each curves represents performance of the lookup policy over changing one $\theta_i$ while $\theta_j=1 \ \ \forall j \neq i$.}
        \label{fig:lookup_param_perfect}
\end{figure}

In the second set of experiments, we evaluate the performance of policies produced by our proposed stochastic approximation algorithms. In Algorithm~\ref{CFA_Grad}, since the stochastic gradient given in \eqnok{grad_barF} is not computable, we use numerical derivatives. Specifically, for each coordinate of the stochastic gradient, we use the finite-difference formula to estimate the corresponding partial derivatives by estimating the objective function at a given $\theta$ and its perturbations for each coordinate. We call this variant of Algorithm~\ref{CFA_Grad}, the Stochastic Numerical Gradient method for the CFA model (SNG-CFA). Since our optimization problem is nonconvex, we implement our algorithms for several different starting points using $N=800$ iterations. We then evaluate the performance of the policy that is produced by averaging over a thousand simulations. We compare the performance of the optimized policies against the base policy using $\theta=1$ in Figures~\ref{fig:lookup_SPSA}. {\color{black}We found that the RMSProp stepsize rule consistently outperformed AdaGrad, so RMSProp is used throughout.
For the SGF-CFA method, we use a mini-batch of sample paths to compute stochastic gradients according to \eqnok{grad_zorth} and then use their average as an estimation for the gradient. To have a fair comparison, we use a mini-batch of size $12$ at each iteration of this algorithm so that its computational cost is similar to that of the SNG-CFA method. While the latter does not have theoretical finite-time convergence guarantees, we found that both algorithms have comparable practical performance.
%While the SGF-CFA method can improve the non-parameterized policy for all starting points, both algorithms have comparable performance, generally speaking.
%Note that performance of both algorithms when started with the unit policy is almost as good as the best case with other random starting points. As mentioned before, this is an optimal policy when we are given perfect forecasts. It now also seems to be a good natural starting point for any given forecast. Moreover, as it can be seen, the SNG-CFA method has better performance than the SGF-CFA method for most of the starting points. The numerical derivatives computed at each iteration of the SNG-CFA method has only a few non-zero coordinates implying that all coordinates of the lookup policy are not updated at each iteration. Due to the interaction between the coordinates, this may result in better performance than updating all coordinates as is at each iteration of the SGF-CFA method. On the other hand, the best performance of the SNG-CFA method is slightly better than that of the SGF-CFA method. However, the run time for each iteration of the former is much larger than that of the latter due to the different number of function evaluations for both methods ($24$ vs. $2$).}
\begin{figure}
        \centering
          \centering
          \includegraphics[width=0.4\linewidth]{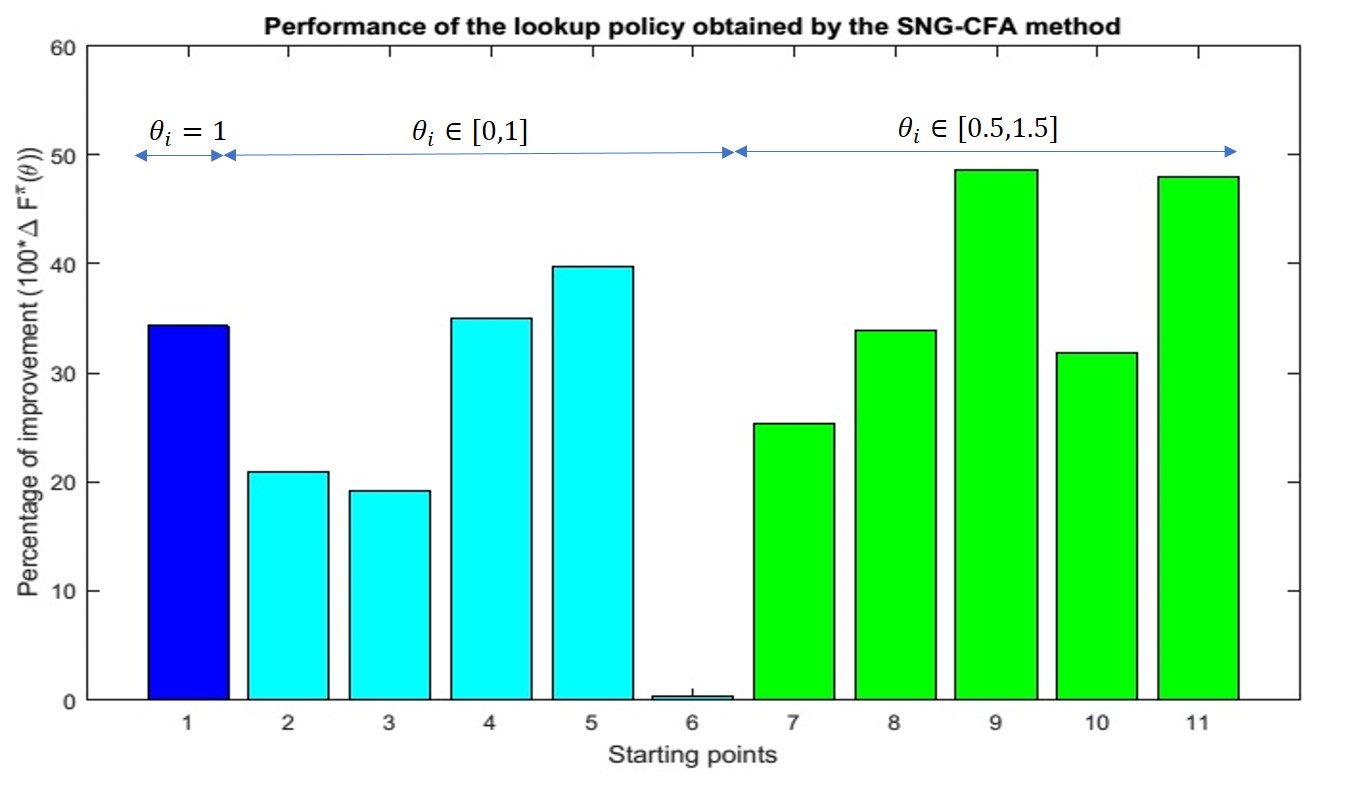}
          \includegraphics[width=0.4\linewidth]{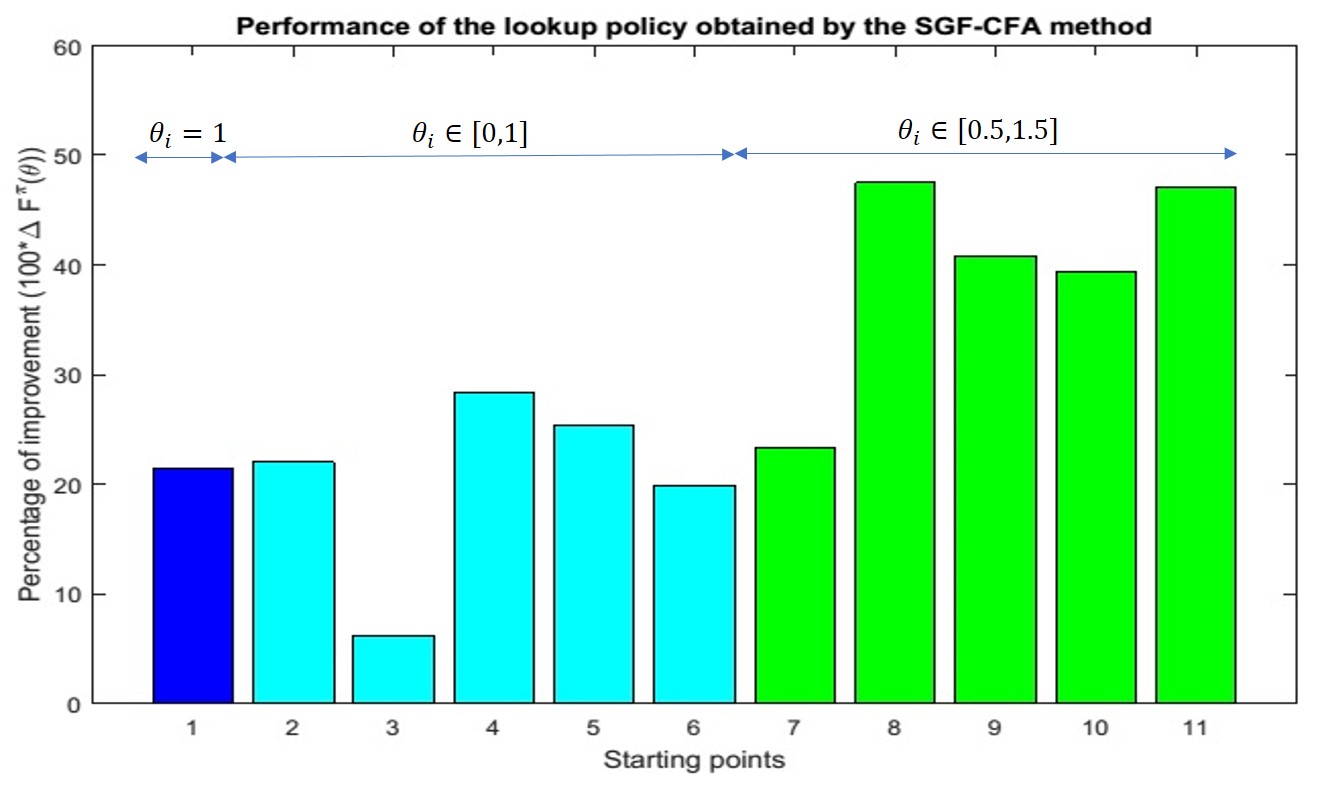}
        \caption{Left: Averaged performance of the output policy of the SNF-CFA method (left) and the SGF-CFA method with batch size of $12$ (right), using RMSProp stepsize for $\sigma^2_{E}=40$ over $1000$ simulations with different starting points.}
        \label{fig:lookup_SPSA}
\end{figure}

In the next set of experiments, we explored the structure of the response surface in our optimization problem
%We first randomly generate $1000$ lookup policies and evaluate the objective function for these policies. We then fit a quadratic regression model to these data points to find the most interactive pairs of coordinates of the lookup policy. Finally, we
by doing a two-dimensional grid search for different pairs of coordinates of $\theta$ in the lookup table representation form, while keeping the other coordinates unchanged. We show the improvement in the objective function over the benchmark policy ($\theta = 1$) in Figure \ref{fig:lookup_Heatmap2} (more are shown in Appendix II). All graphs contain ridges on which changing the coordinates does not improve the policy. While the shape of the ridges can be quite different from one pair of coordinates to another one, most of them share some kind of unimodularity.
\begin{figure}
        \centering
%        \begin{subfigure}{.5\textwidth}
%          \centering
         \includegraphics[width=0.4\textwidth]{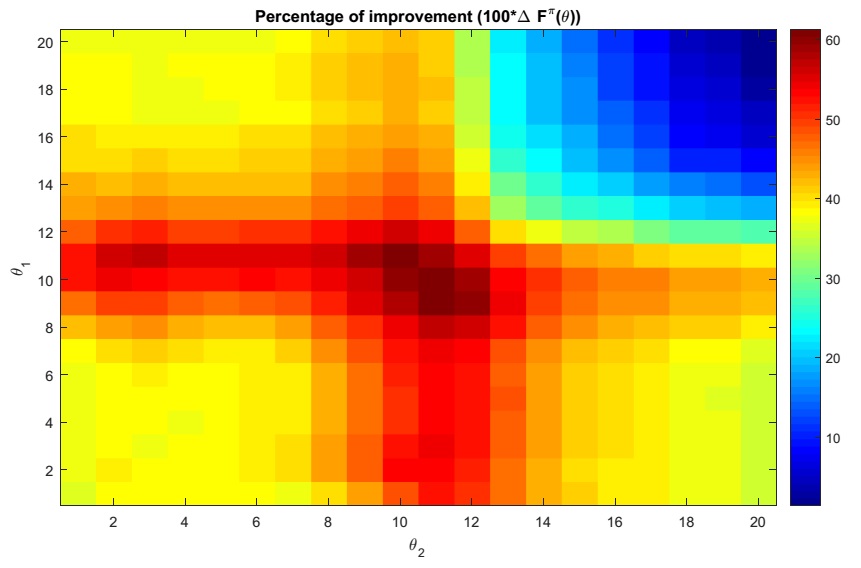}
%        \end{subfigure}%
%        \begin{subfigure}{.5\textwidth}
%          \centering
          \includegraphics[width=0.4\textwidth]{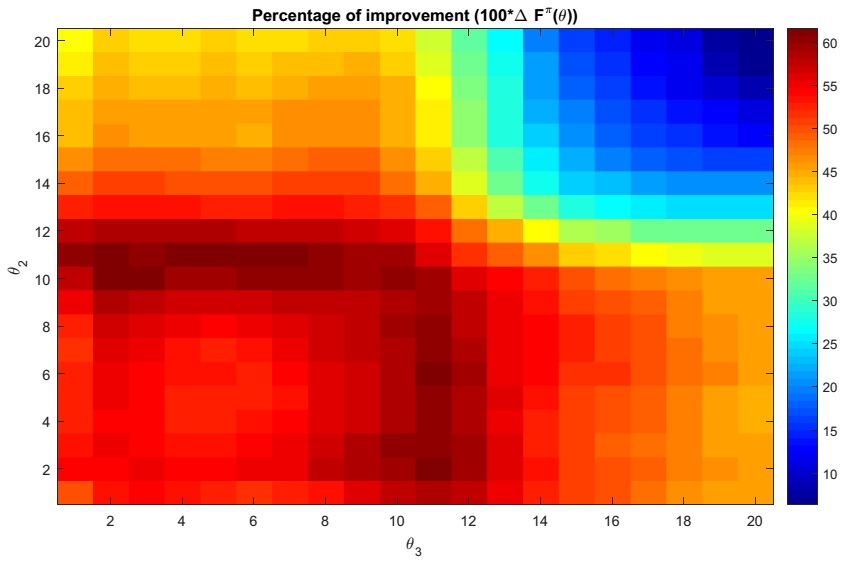}
  %      \end{subfigure}
        \caption{Averaged performance of lookup policies under noisy forecast. Each graph represents improvements over a two-dimensional grid surface of pairs of coordinates of $\theta$.}
        \label{fig:lookup_Heatmap2}
\end{figure}
\section{Conclusion}
This work builds upon a long history of using deterministic optimization models to solve sequential stochastic problems. Unlike other deterministic methods, our class of methods, CFAs, parametrically modify deterministic approximations to account for problem uncertainty. Our particular use of modified linear programs within the framework of the CFA represents a fundamentally new approach for solving multistage stochastic programming problems. Our method allows us to exploit the structural properties of the problem while capturing the complex dynamics of the full base model, rather than accepting the approximations required in a stochastic lookahead model.
%We have demonstrated this class of policies in the context of a complex, time-dependent energy storage problem with forecasts. For our numerical work we selected an energy storage problem that is relatively small to simplify the extensive computational work. However, our methodology is scalable to any problem setting which is currently being solved using a deterministic model.

%An important feature of our approach is that it can handle complex dynamics. For example, we were able to handle the complex hidden semi-Markov model used to represent renewable energy described in section 6.1.  Our methodology would not be affected if this were replaced with any other time series model, or even an observed sample path from history (for which there is no model).
The parametric CFA approach indeed represents an alternative to stochastic lookahead models that represent the foundation of stochastic programming. It requires some intuition into how uncertainty might affect the optimal solution. We would argue that this requirement parallels the design of any parametric statistical model, and hence enjoys a long history.  We believe there are many problems where practitioners have a good sense of how uncertainty affects the solution. % However, further research will be required to determine how well this methodology performs compared to classical stochastic programming models based on scenario trees.
This approach opens up entirely new lines of research.

{\bf{Acknowledgments}}

This research has been supported in part by the National Science Foundation, grant CMMI-1537427.

%%%%%%%%%%%%%%%%%%%%%%%%%%%
%%%%%%%%%%%%%%%%%%%%%%%%%%%
%\section*{References}
%\singlespace

\newpage

%\section*{Appendix I}\label{forecasts}
%\begin{APPENDIX}{}
{\bf{Appendix I}}

\subsection*{$1$. Renewable energy and demand model}
Our model is designed in part to create complex nonstationary behaviors to test the ability of our policy to exploit forecasts while managing uncertainty. We use a series of recursive equations to create a realistic model of the stochastic process describing the generation of renewable energy. In particular, letting $f^E_{t,t'}$ be the forecast of the renewable energy at time $t$ for time $t'$ and assuming that $\{f^E_{0,t'}\}_{t'=0,\ldots,\min(H,T)}$ is given, we define
\beq \label{E_forecast}
f^E_{t+1,t'} = f^E_{t,t'}+\epsilon_{t+1,t'} \qquad t=0,\ldots,T-1, \qquad t'=t+1,\ldots,\min(t+H,T)
\eeq
where $\epsilon_{t+1,t'}$ represents the level of noise whose distribution depends on $f^E_{t,t'}$. To create such noise, we first construct a symmetric matrix $\Sigma \in \bbr^{H \times H}$ such that $\Sigma(i,j) = \sigma_{E}^2 e^{-\alpha|i-j|} \ \ \forall i,j$, where $\sigma_{E}, \alpha>0$ are constant numbers. Indeed, we can manipulate the quality of the forecast by changing $\sigma_{E}$. By construction, $\Sigma$ acts as a covariance matrix representing less correlation between the $i$-th and $j$-th elements when they are far from each other. We then define a normal noise vector as
\[
\bar \epsilon_t = L \cdot Z_t,
\]
where $L$ is the lower triangular Cholesky decomposition of $\Sigma$ and $Z_t \sim {\cal N}(0,I_{H \times H})$. Hence, each element of $\bar \epsilon_t$ has a normal distribution with zero mean and variance $\sigma_{E}^2$ due to the fact that $\Sigma=L \cdot L^\top$ and $\Sigma(i,i)=1$. However, these elements are correlated by construction. To avoid nonnegativity of the forecast, we set
\[
[\epsilon_{t+1,t+1},\ldots,\epsilon_{t+1,\min(t+H,T)}]^\top = F_{EM}^{-1} (\phi(trc^t_H(\bar \epsilon_t))),
\]
where the operator $trc^t_H(y)$ truncates the first $\min(H,T-t)$ elements of vector $y$, $\phi(\cdot)$ is standard normal density function, and $F_{EM}$ is an empirical cumulative distribution function obtained from historical data. The choice of $F_{EM}$ depends on ${f^E}_{t,t'}$ for $t' = t+1, \ldots,\min(t+H,T)$. Figure~\ref{fig:CDF} shows five examples of empirical cumulative distribution functions for the change in wind speed, used in our experiments. It should be mentioned that if ${f^E}_{t,t'}$ becomes negative (by low chance), we just map it to $0$. Finally, after generating all forecasts, we force the observed value of the renewable energy at time $t$ to be ${f^E}_{t,t}$.

\begin{figure}[h]
    \centering
    \includegraphics[width=0.6\textwidth]{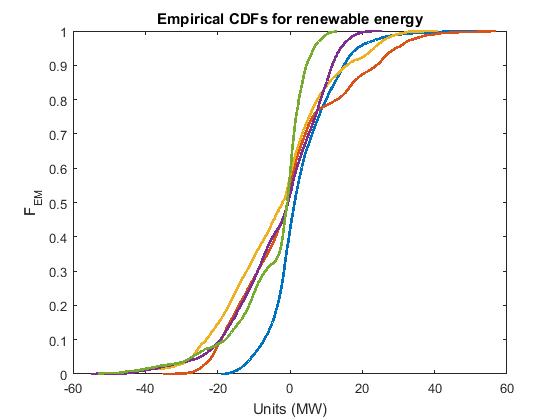}
    \caption{Empirical cumulative distribution functions from a real data set for the renewable energy.}
    \label{fig:CDF}
\end{figure}

%We use a hidden Markov model (\cite{durante2016Markov}) to create a very realistic model of the stochastic process describing the generation of renewable energy and make the amount of energy available from the grid a function of time. This model generates forecast errors based on an underlying crossing time distribution, the consecutive periods of time for which the observed energy produced is above or below the forecast. These errors are modeled using a two-level Markov model with two state variables that evolve on different time scales. The primary state variable, which contains all the pertinent information to approximate the current period's error distribution, evolves at every discrete point of time. The secondary state variable, also known as the crossing state of the system, contains the sign of the error and the duration of how long the sample path has been above or below the forecast. Unlike the primary state variable, this secondary state variable is only updated when forecast errors change signs. Forecast errors are then generated using a distribution selected by a second level Markov model conditioned on the crossing state of the system.

A generated sample of observed renewable energy and its prospective forecast can be viewed in Figure~\ref{fig:renewables}.
%Arrows have been added to identify crossing times.
This is an example of a complex stochastic process that causes problems for stochastic lookahead models.  For example, it is very common when using the stochastic dual dynamic programming (SDDP) to assume interstage independence, which means that $W_t$ and $W_{t+1}$ are independent, which is simply not the case in practice (\cite{shapiro2013risk} and \cite{dupavcova2002comparison}). However, capturing this dynamic in a stochastic lookahead model is quite difficult. Our CFA methodology, however, can easily handle these more complex stochastic models since we only need to be able to simulate the process in the base model.
\begin{figure}[h]
    \centering
    \includegraphics[width=0.6\textwidth]{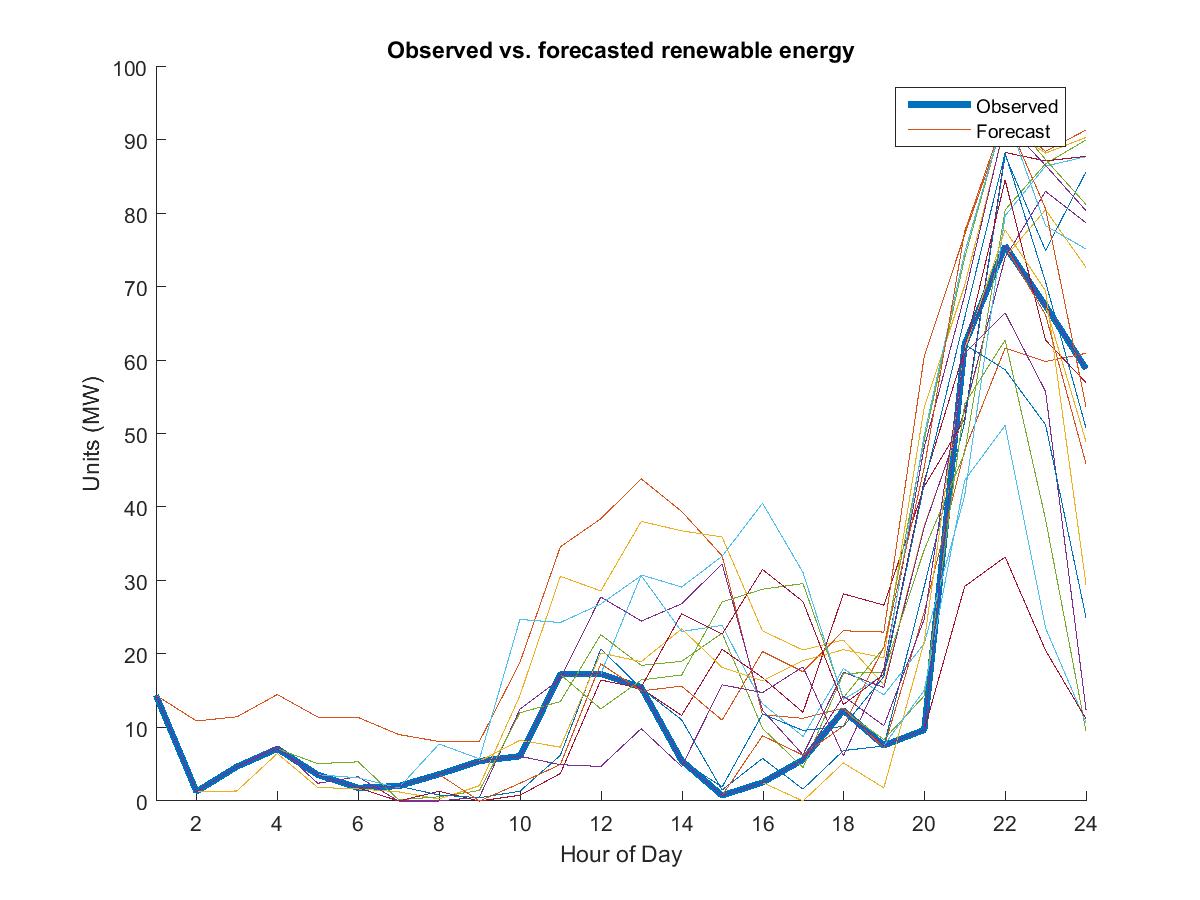}
    \caption{Generated sample of renewable energy ($E_t$)}
    \label{fig:renewables}
\end{figure}
%This allows us to modify the quality of our forecast without modifying the observed stochastic process ($P_t$).
%Different quality forecasts for the same sample path can be seen in figure \ref{fig:renewables_fore}.
%\begin{figure}[h]
%    \centering
%    \includegraphics[width=0.6\textwidth]{images/varying_renewable_observed_MC.jpg}
%    \caption{Forecasts of renewable energy ($E_t$)}
%    \label{fig:renewables_fore}
%\end{figure}
%The amount of energy available from the main grid at $t$, $G_t$ is defined as:
%\begin{equation}
%	G_{t} = \min \bigg\{ \max \left\{ 90 - 50\sin \bigg( \frac{5\pi t}{2T} \bigg), G_{\min} \right\} G_{\max} \bigg\}  %\: \:  \text{ where } \epsilon \sim \mathcal{N}(\mu_g, \sigma_g)
%\end{equation}
%where $G_{\min}$ is the minimum energy always accessible from the grid, $G_{\max}$ is the maximum energy every accessible.

We use the above approach in a backward format to generate demand forecasts. More specifically, assuming that $\{f^D_{t,t}\}_{t=0,\ldots,T}$ is given, we define
\beq \label{D_forecast}
f^D_{t-1,t'} = f^D_{t,t'}+trc^{t-1}_H(\bar \epsilon_{t-1,t'}) \qquad t=T,\ldots,1, \qquad t'=t,\ldots,\min(t+H,T).
\eeq
Note that since the observed demands are set in advance, we can avoid the nonnegativity issue without using the aforementioned inverse CDF of a (uniform) random variable. Indeed, since the observed demands are usually cyclic, their values are specified with a sinusoidal stochastic function:
\begin{equation}
	f^D_{t,t}= D_t = \lceil \max \{0 ,D - \hat D \sin \bigg( \frac{\bar D \pi t}{T} \bigg)+e_t \}\rceil \qquad t=0,1,\ldots,T,
\end{equation}
where $\{e_t\}_{t \ge0}$ are correlated standard normal random variables and $D >\hat D, \bar D$ are positive constants. A generated sample of observed demand and its prospective forecast can be viewed in Figure \ref{fig:demands}.
\begin{figure}[h]
    \centering
    \includegraphics[width=0.6\textwidth]{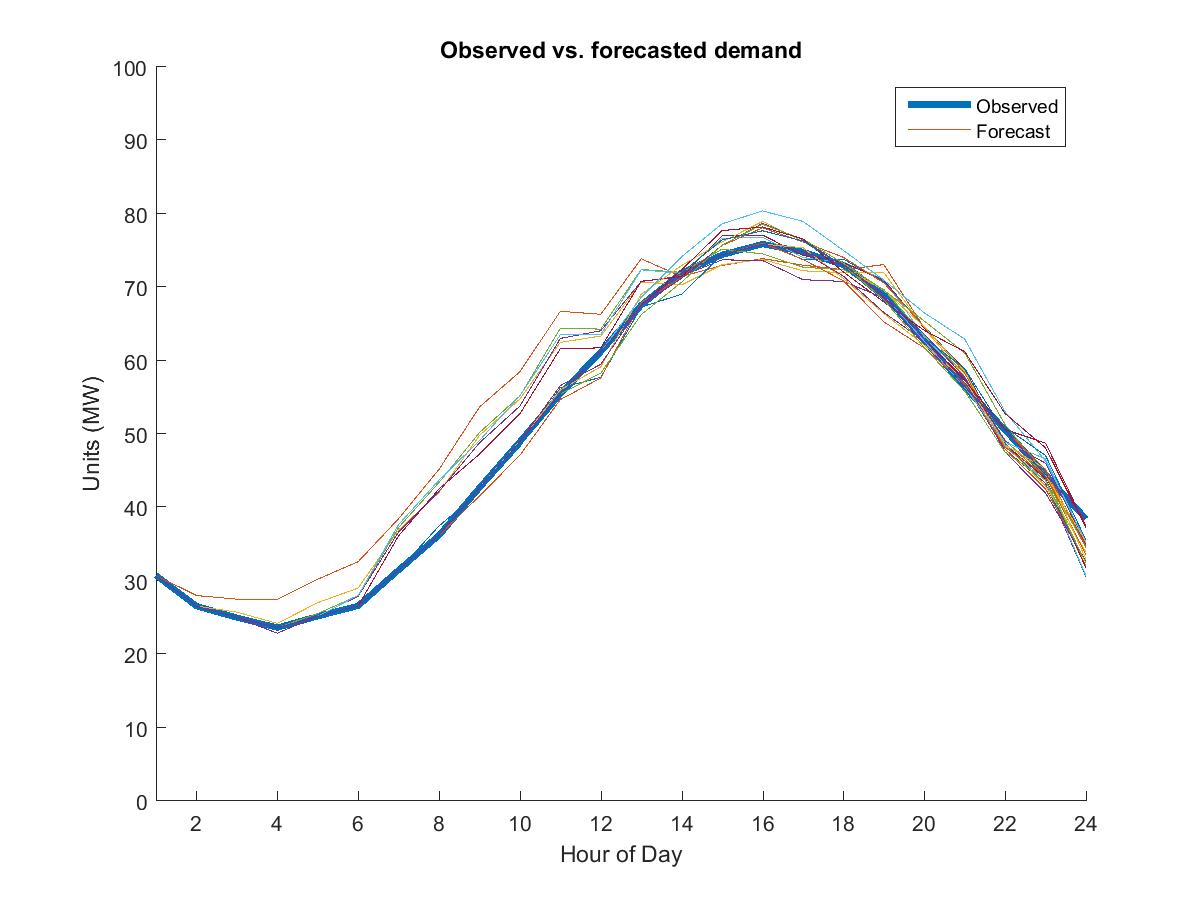}
    \caption{Generated sample of demand ($D_t$)}
    \label{fig:demands}
\end{figure}

\subsection*{$2$. Spot price model}
We assume that spot price of electricity from the grid has a positive correlation with the demand. In particular, we set
\[
f^P_{t,t'} = a+bf^D_{t,t'} \qquad t=0,1,\ldots,T, \qquad t'=t,t+1,\ldots,T,
\]
where $b>0$ and $a \sim {\cal N}(\mu_p,\sigma_p)$ for some $\mu_p,\sigma_p \ge 0$.
%Sample paths of the grid price is displayed in Figure \ref{fig:prices}.
Moreover, the observed and forecasted market prices are fixed and set to the average of all forecasted grid prices assuming a long-term contract with the customer.

\newpage

%\section*{Appendix I}\label{forecasts}
%\begin{APPENDIX}{}
{\bf{Appendix II}}

\subsection*{Extended Numerical Experiments}
In this Section, we provide more numerical experiments about the energy storage problem described in Section~\ref{sec_energy}. In the first one, we evaluate the performance of the constant forecast parameterization for different levels of uncertainty in forecasting the supply from renewable energy. In this case, the optimization problem is one dimensional and hence, we use a grid search to find the optimal policy. As shown in Figure~\ref{fig:single_param}, performance of the constant parameterization policy is improved by increasing the value of $\theta$ to one point and then decreased. It is also worth noting that under perfect forecasts, $\theta=1$ is the optimal value as mentioned before.

\begin{figure}[htp]
%                  \centering
\begin{center}
                  \includegraphics[width=.45\linewidth]{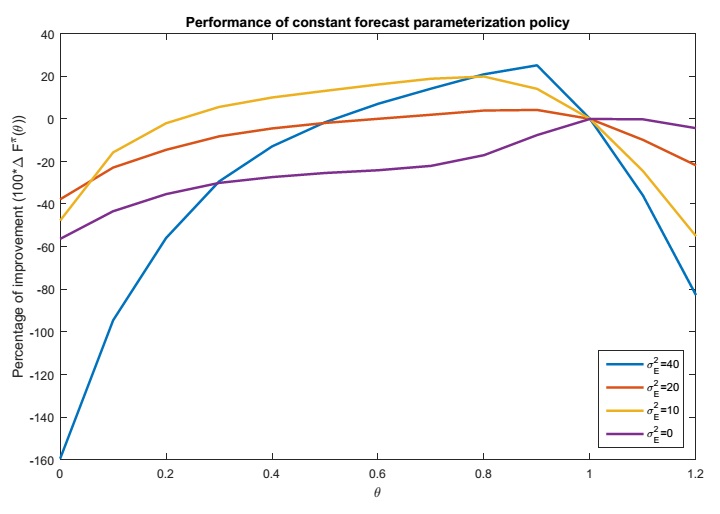}
                  \caption{Averaged performance of constant parameterization policy over $1000$ simulation.}
                  \label{fig:single_param}	
\end{center}
\end{figure}

In the next set of experiments, we provide more graphs in Figure~\ref{fig:lookup_Heatmap2} about the behaviour
of the objective function in terms of improvement over the benchmark policy for different pairs of coordinates of $\theta$ in the lookup up table representation form.

\begin{figure}
        \centering
%        \begin{subfigure}{.5\textwidth}
%          \centering
         \includegraphics[width=0.45\linewidth]{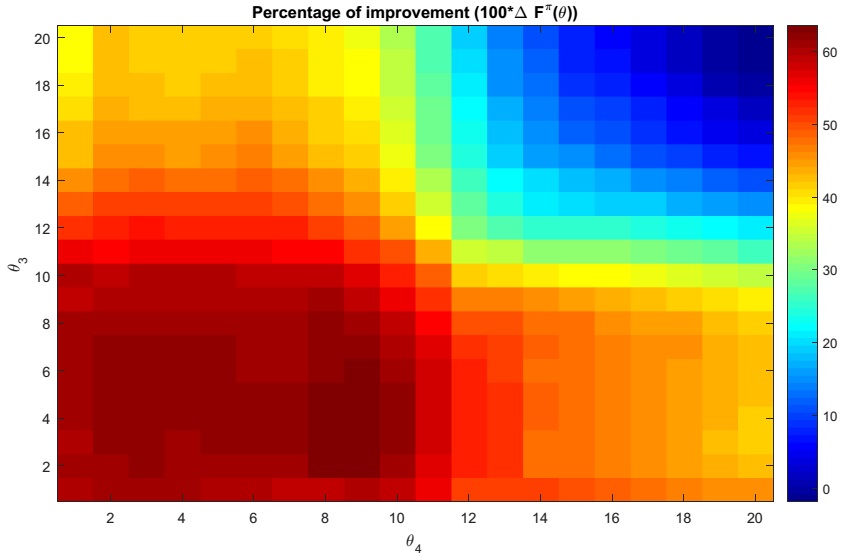}
         \includegraphics[width=0.45\textwidth]{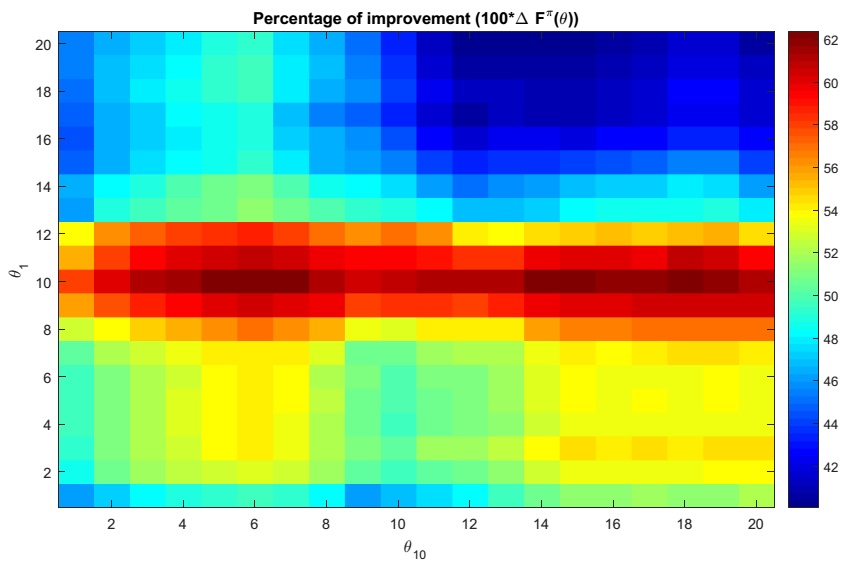}
%        \end{subfigure}%
%        \begin{subfigure}{.5\textwidth}
%          \centering
          \includegraphics[width=0.45\textwidth]{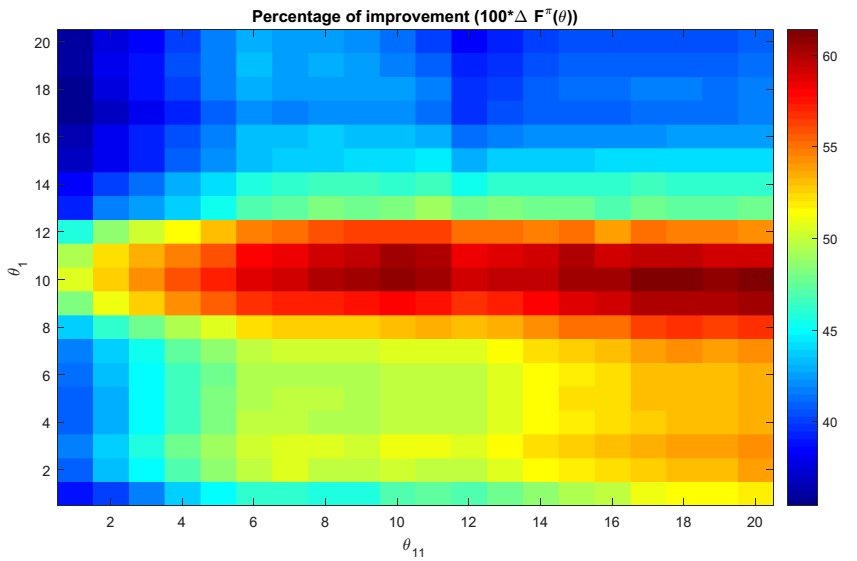}
  %      \end{subfigure}
%        \begin{subfigure}{.5\textwidth}
%          \centering
        \includegraphics[width=0.45\linewidth]{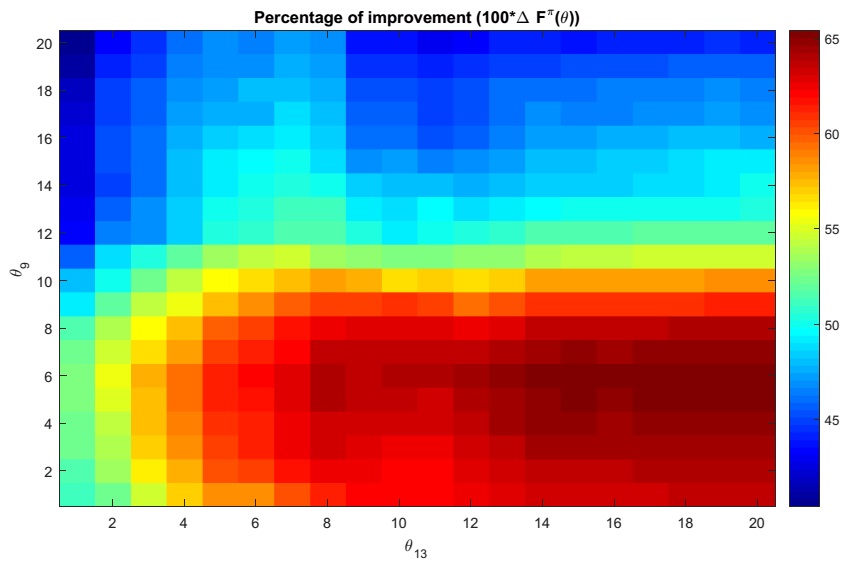}
        \caption{Averaged performance of lookup policies under noisy forecast. Each graph represents improvements over a two-dimensional grid surface of pairs of coordinates of $\theta$.}
        \label{fig:lookup_Heatmap3}
\end{figure}

\end{document}